\documentclass[a4paper,reqno]{amsart}
\usepackage{geometry}
\usepackage{mathrsfs}
\usepackage{syntonly}
\usepackage{amsmath}
\usepackage{amsthm}
\usepackage{amsfonts}
\usepackage{amssymb}
\usepackage{latexsym}
\usepackage{amscd,amssymb,amsopn,amsmath,amsthm,graphics,amsfonts,mathrsfs,accents,enumerate,verbatim,calc}
\usepackage[dvips]{graphicx}
\usepackage[colorlinks=true,linkcolor=blue,citecolor=blue]{hyperref}
\input xy
\xyoption{all}
\def\del{\delta}
\usepackage{setspace}

\usepackage{color}

\numberwithin{equation}{section}

\theoremstyle{plain}

\newtheorem{thm}{Theorem}[section]
\newtheorem{cor}[thm]{Corollary}
\newtheorem{lem}[thm]{Lemma}
\newtheorem{prop}[thm]{Proposition}

\newtheorem{defn}[thm]{Definition}
\newtheorem{exm}[thm]{Example}

\newtheorem{rem}[thm]{Remark}

\newdir{ >}{{}*!/-5pt/\dir{>}}

\newcommand{\Hom}{\operatorname{Hom}\nolimits}

\newcommand{\End}{\operatorname{End}\nolimits}

\newcommand{\op}{\operatorname{op}\nolimits}

\renewcommand{\mod}{\mathsf{mod}\hspace{.01in}}

\newcommand{\M}{\mathcal M}

\newcommand{\B}{\mathcal B}
\newcommand{\uB}{\underline{\B}}

\newcommand{\U}{\mathcal U}
\newcommand{\V}{\mathcal V}
\newcommand{\A}{\mathcal A}
\newcommand{\W}{\mathcal W}
\newcommand{\h}{\mathcal H}

\newcommand{\s}{\mathcal S}
\newcommand{\T}{\mathcal T}

\newcommand{\D}{\mathcal D}

\newcommand{\N}{\mathcal N}

\newcommand{\X}{\mathscr X}

\newcommand{\C}{\mathcal C}

\newcommand{\E}{\mathbb E}
\newcommand{\EE}{\mathbb E}

\newcommand{\svecv}[2]{\left(\begin{smallmatrix}
      #1 \\
      #2
    \end{smallmatrix}\right)}

\newcommand{\svech}[2]{\left(\begin{smallmatrix}
      #1 & #2
\end{smallmatrix}\right)}

\renewcommand{\emph}{\textit}
\renewcommand{\phi}{\varphi}

\newcommand{\add}{\mathsf{add}\hspace{.01in}}
\begin{document}

\title[Largest exact structures and almost split sequences on hearts]{Largest exact structures and almost split sequences on hearts of twin cotorsion pairs}\footnote{\hspace{1em}Yu Liu is supported by the National Natural Science Foundation of China (Grant No. 12171397).  Wuzhong Yang is supported by the National Natural Science Foundation of China (Grant Nos. 12271321 and 11971384).
Panyue Zhou is supported by the Hunan Provincial Natural Science Foundation of China (Grant No. 2023JJ30008) and by the National Natural Science Foundation of China (Grant No. 11901190). }
\author{Yu Liu, Wuzhong Yang and Panyue Zhou}
\address{ School of Mathematics and Statistics, Shaanxi Normal University, 710062 Xi'an, Shannxi, P. R. China}
\email{recursive08@hotmail.com}

\address{School of Mathematics, Northwest University, 710127 Xi'an, Shaanxi, P. R. China}
\email{mengxiankun061@qq.com}

\address{School of Mathematics and Statistics, Changsha University of Science and Technology, 410114 Changsha, Hunan, P. R. China}
\email{panyuezhou@163.com}

\thanks{The authors would like to thank Professor Bin Zhu for helpful discussions.}
\begin{abstract}
Hearts of cotorsion pairs on extriangulated categories are abelian categories. On the other hand, hearts of twin cotorsion pairs are not always abelian. They were shown to be semi-abelian by Liu and Nakaoka. Moreover, Hassoun and Shah proved that they are quasi-abelian under certain conditions. In this article, we first show that the heart of any twin cotorsion pair has a largest exact category structure and is always quasi-abelian. We also provide a sufficient and necessary condition for the heart of a twin cotorsion pair being abelian. Then by using the results we have got, we investigate the almost split sequences in the hearts of twin cotorsion pairs.  Finally, as an application, we show that a Krull-Schmidt, Hom-finite triangulated category has a Serre functor whenever it has a cluster tilting object.
\end{abstract}
\maketitle

\section{Introduction}

Cluster tilting theory gives a way to construct abelian categories from triangulated categories and exact categories: one can pass from the original categories to abelian quotient categories by factoring out cluster tilting subcategories. On triangulated categories, such results were shown by Buan, Marsh and Reiten \cite[Theorem 2.2]{BMR} for cluster categories, by Keller and Reiten \cite[Proposition 2.1]{KR} for $2$-Calabi-Yau case and then by Koenig and Zhu \cite[Theorem 3.3]{KZ} for the general case. Demonet and Liu \cite[Theorem 3.2]{DL} proved an analogous result for exact categories with enough projectives.

The notion of cotorsion pair was first introduced by Salce \cite{Sa}. It was defined originally on abelian groups, and then on exact and triangulated categories.  All the results mentioned above have a unified cotorsion pair version, which is more general without losing too much information. We will explain in detail later.

On triangulated categories, the concept of cotorsion pair is an analog of torsion pair: $(\U,\V)$ is a cotorsion pair if and only if $(\U,\V[1])$ is a torsion pair. We recall the definition of cotorsion pairs on triangulated categories for the convenience of the readers:
Let $\U, \V$ be two full subcategories in a triangulated category $\C$ with shift functor $[1]$. The pair  $(\U,\V)$ is called a \emph{cotorsion pair} on $\C$ if the following conditions are satisfied:
\begin{enumerate}
 \item[(1)]  $\Hom_\C(\U,\V[1])=0$.
 \vspace{1mm}
\item[(2)] Any object $C\in \C$ admits a triangle $V\to U\to C\to V[1]$ with $U\in \U$ and $V\in \V$.
\end{enumerate}
$t$-structures and co-$t$-structures on $\C$ can be realized as special kinds of cotorsion pairs in the following way:
\begin{itemize}
\item[(i)] a cotorsion pair $(\U,\V)$ can be called a $t$-structure if $\U[1]\subseteq \U$;
\item[(ii)] a cotorsion pair $(\U,\V)$ can be called a co-$t$-structure if $\U[-1]\subseteq \U$.
\end{itemize}

Nakaoka \cite{N1} defined the hearts of cotorsion pairs, which is a generalization of the hearts of $t$-structures. He also generalized the well-known result for the hearts of $t$-structures in \cite{BBD}, showing that the hearts of cotorsion pairs are abelian.

The notion of extriangulated category was introduced by Nakaoka and Palu \cite{NP}, which is a simultaneous generalization of triangulated categories and exact categories. Cotorsion pairs \cite{NP}, their hearts (which are shown to be abelian) \cite{LN} and cluster tilting subcategories \cite{ZZ2} can be defined on extriangulated categories, which are generalizations of the same concepts on both triangulated categories and exact categories.

Any cluster tilting subcategory $\T$ on an extriangulated category $\B$ admits a cotorsion pair $(\T,\T)$, and the ideal quotient $\B/\T$ is just the heart of $(\T,\T)$, hence is abelian. This is the unified version of all the results we mentioned at the beginning.

In \cite{N2}, Nakaoka introduced a generalized concept called twin cotorsion pairs: a pair of cotorsion pairs $((\s,\T),(\U,\V))$ satisfying the condition $\s\subseteq \U$. Note that any cotorsion pair $(\U,\V)$ can be realized as a twin cotorsion pair $((\U,\V),(\U,\V))$. He then defined the hearts of twin cotorsion pairs. It is also a generalization of the hearts of cotorsion pairs in the sense that the heart of the twin cotorsion pair $((\U,\V),(\U,\V))$ is just the heart of $(\U,\V)$. Later these concepts were generalized to extriangulated categories \cite{NP,LN}. An interesting example of twin cotorsion pairs was hidden in \cite{BM}. In that paper, Buan and Marsh considered a rigid object $T$ on a Krull-Schmidt, Hom-finite triangulated category $\C$ with Serre functor. In fact, $T$ admits a twin cotorsion pair $((\add T[1],\X_T),(\X_T,\V))$ (where $\X_T:=\{M\in\C\mid {\rm Hom}_{\C}(T,M)=0\}$) and $\C/\X_T$ is the heart of this twin cotorsion pair. $\C/\X_T$ is not abelian in general, but they showed that the localization of  $\C/\X_T$ with respect to the regular morphisms on it is equivalent to $\mod \End_\C(T)^{\op}$. Another interesting example is the following: by the results in \cite{LN}, $n$-cluster ($n>2$) tilting subcategory $\T$ on an extriangulated category $\B$ with enough projectives and enough injectives induces a twin cotorsion pair $((\T,\M),(\M,\N))$.

Generally speaking, hearts of twin cotorsion pairs are not always abelian (see Example \ref{ex1} in Section 4). Some ``if and only if" conditions for the hearts being abelian are discussed for the twin cotorsion pairs induced by $n$-cluster ($n>2$) tilting subcategories (see \cite{L2} and \cite{HZ}). It was shown in \cite[Theorem 2.32]{LN} that the heart of any twin cotorsion pair is semi-abelian (same results were shown before for triangulated categories \cite{N2} and exact categories \cite{L1}). Later, Hassoun and Shah \cite[Theorem 4.2]{HS} proved that the hearts are quasi-abelian under certain conditions.

The notion of semi-abelian category was introduced by
Janelidze, M\'{a}rki and Tholen \cite{JMT} in order to capture typical
algebraic properties valid for groups, rings and algebras.
Semi-abelian categories provide a categorical approach to the isomorphism and decomposition theorems of group theory, to general theories of radicals and commutators, and to homology theory of non-abelian structures. 

Rump introduced a class of additive categories close to abelian
categories which are called quasi-abelian categories in \cite{R1} (note that Rump used the term ``almost abelian"
instead of ``quasi-abelian"). There are plenty of concrete quasi-abelian
categories, such as various categories of topological abelian groups, topological
vector spaces and lattices over orders.

\vspace{1mm}

We recall the definitions of semi-abelian category and quasi-abelian category.

\begin{defn}
An additive category is called preabelian if any morphism admits a kernel and a cokernel. A preabelian category $\mathcal A$ is called \emph{left semi-abelian} (resp. \emph{left quasi-abelian}) if in any pull-back diagram
$$\xymatrix{
A \ar[r]^{\mathbf{a}} \ar[d]_{\mathbf{b}} &B \ar[d]^{\mathbf{c}}\\
C \ar[r]_{\mathbf{d}} &D}
$$
the morphism $\mathbf{a}$ is an epimorphism (resp. a cokernel) whenever $\mathbf{d}$ is a cokernel. Dually we can define \emph{right semi-abelian} (resp. \emph{right quasi-abelian}) categories. $\mathcal A$ is called \emph{semi-abelian} (resp. \emph{quasi-abelian}) if it is both left and right semi-abelian (resp. quasi-abelian).
\end{defn}

By definition, any quasi-abelian category is semi-abelian. 
We have the following relation:
$$\{ \mbox{Abelian Cat.} \}\subset\{ \mbox{Quasi-Abelian Cat.} \}\subset\{ \mbox{Semi-Abelian Cat.} \}\subset\{ \mbox{Preabelian Cat.} \}.$$
For more details about these categories, see \cite{R1}. Note that semi-abelian category is not always quasi-abelian, since counterexamples have been found in \cite{BD,R2}. But how about the hearts of twin cotorsion pairs?
\vspace{1mm}

According to \cite[Theorem 3.3]{SW} (see \cite[Theorem 3.5]{C} for a generalized result), every preabelian category $\mathcal A$ has a maximal exact category structure, maximal means that all exact structures on $\mathcal A$ are contained within it. Since the hearts of twin cotorsion pairs are preabelian, they all have maximal exact structures. But we show that these maximal structures on the hearts are more special. This observation helps us to find that the hearts of twin cotorsion pairs are quasi-abelian. In fact, we have the following theorem.

\begin{thm}{\rm (see Theorem \ref{ck} and Theorem \ref{main} for details)}
Let $(\B,\EE,\mathfrak{s})$ be an extriangulated category. The heart of any twin cotorsion pair on $\B$ has a largest exact category structure in the sense that any kernel-cokernel pair is a short exact sequence. Moreover, this heart is quasi-abelian.
\end{thm}

Although the hearts of twin cotorsion pairs are quasi-abelian, as mentioned above, they are not always abelian. We find that if the heart of a twin cotrsion pair $((\s,\T),(\U,\V))$ is abelian, it equals to the intersection of the heart of $(\s,\T)$ and the heart of $(\U,\V)$ (see Proposition \ref{prop1} for details). Moreover, By using the results we have got, we provide a sufficient and necessary condition for the heart of a twin cotorsion pair being abelian. The second main result is the following.

\begin{thm}{\rm (see Theorem \ref{main2} for details)}
Let $(\B,\EE,\mathfrak{s})$ be an extriangulated category and $((\s,\T),(\U,\V))$ be a twin cotorsion pair. The heart of $((\s,\T),(\U,\V))$ is abelian if and only if any epimorphism $\alpha: B\to C$ in the heart admits an $\EE$-triangle
$$\xymatrix{B' \ar[r]^-{p'} &C'\ar[r] &S \ar@{-->}[r] &}$$
such that $S\in \s$ and the image of $p'$ in the heart equals to $\alpha$.
\end{thm}

Auslander-Reiten theory, which was established by series of important works \cite{AR1, AR2}, is crucial for the representation theory of algebra. Almost split sequence plays a very important role in this theory. Later, many results in the theory were shown on more abstract categories, for example, arbitrary abelian categories. Recently, Auslander-Reiten theory for extriangulated categories has been established in \cite{INP}, we can discuss almost split sequences on extriangulated categories. Since the hearts of twin cotorsion pairs are now shown to be quasi-categories (they have largest exact structures and hence are also extriangulated categories), moreover, the exact structures of the hearts are inherited from the original categories, by using the theory we have shown, we can find the relation between the almost split sequences in the hearts and the original categories. (Note that Auslander-Reiten theory in quasi-abelian categories has been investigated in \cite{Sh}, but in this paper, we will focus on the relation between almost split sequences in different categories.)

\begin{thm}{\rm (see Definition \ref{as} and Theorem \ref{main3} for details)}
Let $(\B,\EE,\mathfrak{s})$ be a Krull-Schmidt extriangulated category and $((\s,\T),(\U,\V))$ be a twin cotorsion pair. There exists a one-to-one correspondence between the almost split sequences in the heart of $((\s,\T),(\U,\V))$ and the almost split sequences in a subcategory of $(\B,\EE,\mathfrak{s})$ related to this heart.
\end{thm}

\medskip

As an application of our results, we go back to a triangulated category with a cluster tilting object. As mentioned at the beginning, for a Krull-Schmidt, Hom-finite triangulated category $\C$ with a cluster tilting subcategory $\T$, we have an abelian ideal quotient $\C/\T$. If $\T=\add T$, then we have $\C/\T\simeq \mod \End_\C(T)$, hence $\C/\T$ has Auslander-Reiten sequences. Based on these facts, by using the results we have got, we show the following theorem.

\begin{thm}
Let $\C$ be a Krull-Schmidt, Hom-finite, $k$-linear triangulated category. If $\C$ has a cluster tilting object, then $\C$ has a Serre functor.
\end{thm}

This paper is organized as follows. In Section 2, we review some elementary concepts and properties of extriangulated categories.
In Section 3, we prove that the hearts of twin cotorsion pairs have largest exact category structures and are quasi-abelian. In Section 4, we provide a sufficient and necessary condition when the hearts of twin cotorsion pairs become abelian. In Section 5, we investigate the relation between the almost split sequences in the hearts and $(\B,\EE,\mathfrak{s})$. In Section 6, we show some results associated with  a cluster tilting object on a triangulated category as an application.

\section{Preliminaries}
In this section, we collect  some terminologies and basic properties of extriangulated categories
which we need later. They will also help the readers to have an overlook at extriangulated categories. We omit the precise definition (\cite[Definition 2.12]{NP}), and for more details, we refer to \cite[Sections 2 and 3]{NP}.

An extriangulated category is an additive category $\B$ equipped with an additive bifunctor
$$\mathbb{E}: \B^{\rm op}\times \B\rightarrow \mathbf{Ab},$$
where $\mathbf{Ab}$ is the category of abelian groups and a correspondence $\mathfrak{s}$.

For any pair of objects $A,C\in\B$, any element $\del\in\EE(C,A)$ is called an $\EE$-extension. For any morphisms $a\in\Hom_\B(A,A')$ and $c\in\Hom_\B(C',C)$,  $\EE(C,a)(\del)\in\EE(C,A')\ \ \text{and}\ \ \EE(c,A)(\del)\in\E(C',A)$ are simply denoted by $a_{\ast}\del$ and $c^{\ast}\del$, respectively. Let $\del\in \EE(C,A)$ and $\del'\in \EE(C',A')$ be any pair of $\E$-extensions. A {\it morphism} $(a,c)\colon\del\to{\delta}{'}$ of extensions is a pair of morphisms $a\in\Hom_\B(A,A')$ and $c\in\Hom_\B(C,C')$ in $\B$, satisfying the equation
$a_{\ast}\del=c^{\ast}{\delta}{'}$.

We say two sequences $A\xrightarrow{x} B\xrightarrow{y} C$, $A\xrightarrow{x'} B'\xrightarrow{y} C'$ are equivalent if they admit a commutative diagram.
$$\xymatrix@C=1cm@R=0.3cm{
&B \ar[dr]^y \ar[dd]^-{\cong}\\
A \ar[dr]_{x'} \ar[ur]^x &&C\\
&B' \ar[ur]_{y'}
}
$$
The sequences which are equivalent to each other form an equivalence class, we use $[A\xrightarrow{x} B\xrightarrow{y} C]$ to denote such class.

In extriangulated category $\B$, every $\EE$-extension $\delta\in\mathbb{E}(C, A)$ associates with an equivalence class of sequences by the correspondence $\mathfrak{s}$:
$$\mathfrak{s}(\delta)=\xymatrix@C=0.8cm{[A\ar[r]^x &B\ar[r]^y&C]}.$$
Moreover, if we have two extensions $\del\in \EE(C,A)$, $\del'\in \EE(C',A')$ such that
$$\mathfrak{s}(\delta)=\xymatrix@C=0.8cm{[A\ar[r]^x &B\ar[r]^y&C]},\quad \mathfrak{s}(\delta')=\xymatrix@C=0.8cm{[A'\ar[r]^{x'} &B'\ar[r]^{y'}&C']}$$
and an equation $a_{\ast}\del=c^{\ast}{\delta}{'}$, then we can get the following commutative diagram:
$$\xymatrix{
A\ar[r]^x \ar[d]_a &B\ar[r]^y \ar[d]^{\exists ~b} &C \ar[d]^c\\
A'\ar[r]_{x'} &B'\ar[r]_{y'}&C'.
}$$

\smallskip

\begin{defn}\label{dein}\cite[Definitions 2.17, 2.19 and 3.23]{NP}
Let $(\B,\EE,\mathfrak{s})$ be an extriangulated category.
\begin{itemize}
\setlength{\itemsep}{2.5pt}
\item[{\rm (1)}] Let $\del\in\EE(C,A)$. If $\mathfrak{s}(\delta)=[A\xrightarrow{~x~}B\xrightarrow{~y~}C]$, we call the following sequence
$$A\overset{x}{\longrightarrow}B\overset{y}{\longrightarrow}C\overset{\delta}{\dashrightarrow}$$
 an $\EE$-triangle. We also say $\delta$ is realized by $A\overset{x}{\longrightarrow}B\overset{y}{\longrightarrow}C$. We can omit $``\delta"$ in the $\EE$-triangle if it is not used in the argument.

\item[{\rm (2)}] An object $P\in\B$ is called {\it projective} if
for any $\EE$-triangle $A\overset{x}{\longrightarrow}B\overset{y}{\longrightarrow}C \dashrightarrow$ and any morphism $c:P\to C$, there exists a morphism $b:P\to B$ satisfying $yb=c$. Dually we can define injective objects. An object is called  a projective-injective object if it is both projective and injective.
%

\item[{\rm (3)}] Let $\s$ be a subcategory of $\B$. We say $\s$ is {\it extension closed}
if in any $\EE$-triangle $A\rightarrow B\rightarrow C\dashrightarrow$ with $A,C\in\s$, we have $B\in\s$.

\end{itemize}
\end{defn}

The following property is very important for extriangulated categories (see \cite[Proposition 3.3]{NP} for details).

\begin{prop}\label{exs}
Let $(\B,\EE,\mathfrak{s})$ be an extriangulated category. For any $\EE$-triangle
$$A\overset{x}{\longrightarrow}B\overset{y}{\longrightarrow}C\dashrightarrow,$$
we have the following two short exact sequences:
\begin{itemize}
\item[(1)] $\Hom_\B(C,-)\xrightarrow{(y,-)}\Hom_\B(B,-)\xrightarrow{(x,-)}\Hom_\B(A,-)\longrightarrow \EE(C,-)$;
\item[(2)] $\Hom_\B(-,A)\xrightarrow{(-,x)}\Hom_\B(-,B)\xrightarrow{(-,y)}\Hom_\B(-,C)\longrightarrow \EE(-,A)$.
\end{itemize}
\end{prop}
\medskip

\begin{rem}\cite[Remark 2.18]{NP}\label{rem}
Any extension closed subcategory $\M$ of an extriangulated category $(\B,\EE,\mathfrak{s})$ has a natural extriangulated category structure $(\M,\EE|_{\M},\mathfrak{s}|_{\M})$  inheriting from $(\B,\EE,\mathfrak{s})$,
where $\EE|_{\M}$ is the restriction of $\EE$ onto ${\M^{\rm op}\times \M}$
and  $\mathfrak{s}|_{\M}$ is the restriction of $\mathfrak{s}$  onto $\M$.
\end{rem}

Note that any extension closed subcategory of a triangulated category is an extriangulated category and any exact category is also an extriangulated category. Since we also need to deal with exact categories later, we recall some concepts about exact categories (there are several notions of exact categories, the one we use in this article attributed to Quillen \cite{Q}).

Let $\mathcal A$ be an additive category. A sequence $A\xrightarrow{x} B\xrightarrow{y} C$ in $\mathcal A$ is called a \emph{kernel-cokernel pair} if $x$ is the kernel of $y$ and $y$ is the cokernel of $x$. An exact category structure on $\mathcal A$ is a pair $(\mathcal A,\mathfrak{S})$ where $\mathfrak{S}$ is a class of kernel-cokernel pairs closed under isomorphisms, satisfying certain axioms (see \cite[Definition 2.1]{Bu} for details). We call a kernel-cokernel pair in $\mathfrak{S}$ a \emph{short exact sequence}.  We say $(\mathcal A,\mathfrak{S})$ is the largest exact category structure if any kernel-cokernel pair belongs to $\mathfrak{S}$.



In this article, let $(\B,\EE,\mathfrak{s})$ be an extriangulated category.
When we say that $\C$ is a subcategory of $\B$, we always assume that $\C$ is full and closed under isomorphisms.

We recall the definition of a cotorsion pair on $\B$.

\begin{defn}\cite[Definition 2.1]{NP}
Let $\U$ and $\V$ be two subcategories of $\B$ which are closed under direct summands. We call $(\U,\V)$ a \emph{cotorsion pair} if the following conditions are satisfied:
\begin{itemize}
\item[(a)] $\EE(\U,\V)=0$.
\smallskip

\item[(b)] For any object $B\in \B$, there exist two $\EE$-triangles
\begin{align*}
V_B\rightarrow U_B\rightarrow B{\dashrightarrow},\quad
B\rightarrow V^B\rightarrow U^B{\dashrightarrow}
\end{align*}
satisfying $U_B,U^B\in \U$ and $V_B,V^B\in \V$.
\end{itemize}

\end{defn}

By the definition of a cotorsion pair, we can conclude the following result.

\begin{lem}\label{L1}\cite[Remark 2.2]{LN}
Let $(\U,\V)$ be a cotorsion pair in $\B$.
\begin{itemize}
\item[(a)] $\V=\{ X\in \B \text{ }|\text{ } \EE(\U,X)=0\}$.
\item[(b)]  $\U=\{ Y\in \B \text{ }|\text{ } \EE(Y,\V)=0\}$.
\item[(c)] $\U$ and $\V$ are closed under extensions.

\end{itemize}
\end{lem}

%
%

By Lemma \ref{L1}, we can get the following corollary.

\begin{cor}
For a pair of cotorsion pairs $((\s,\T),(\U,\V))$, the following conditions are equivalent:
$${\rm (1)} ~\s\subseteq \U; \quad {\rm (2)}~ \V\subseteq \T; \quad {\rm (3)}~ \EE(\s,\V)=0.$$
\end{cor}

\begin{defn}\cite[Definition 4.12]{NP}
A pair of cotorsion pairs $((\s,\T),(\U,\V))$ is called a twin cotorsion pair if $\s \subseteq \U$.
\end{defn}

\begin{rem}
A cotorsion pair $(\U,\V)$ can be realized as a twin cotorsion pair $((\U,\V),(\U,\V))$.
\end{rem}

Given two objects $A$ and $B$ of $\B$, we denote by $\X(A,B)$ the set of morphisms from  $A$ to $B$ which factor through objects lie in a subcategory $\X$. $\X(A,B)$ is a subgroup of $\Hom_\B(A,B)$, and the family of
these subgroups forms an ideal of $\B$. Then we
have a category $\B/\X$ whose objects are the objects of
$\B$ and  the set of morphisms from $A$ to $B$ is $\Hom_\B(A,B)/\X(A,B)$. Such category is called the \emph{ideal quotient category} of $\B$ by $\X$. For a subcategory $\C\supseteq \X$, we also denote by $\C/\X$ the ideal quotient category of $\C$ by $\X$.

In the rest of the paper, let $((\s,\T),(\U,\V))$ be a twin cotorsion pair and $\W=\T\cap \U$. For convenience, we denote by $\uB$ the ideal quotient $\B/\W$. For any morphism $f\in \Hom_{\B}(A,B)$, we denote its image in $\Hom_{\uB}(A,B)$ by $\underline f$. For any subcategory $\D$ of $\B$, we denote by $\underline \D$ the image of $\D$ in $\uB$ by the canonical quotient functor $\pi:\B\to \uB$.

Now we recall the definition of  the heart of a twin cotorsion pair from \cite[Definitions 2.5 and 2.6]{LN}.

\begin{defn}\label{He}
For a twin cotorsion pair $((\s,\T),(\U,\V))$, let
\begin{itemize}
\item[(i)] $\B^+=\{X\in \B \text{ }|\text{ } \mbox{there exists an}~\text{ } \EE\text{-triangle } V\to W\to X\dashrightarrow \text{, }V\in \V \text{ and }W\in \W \}$,
\item[(ii)] $\B^-=\{Y\in \B \text{ }|\text{ } \mbox{there exists an}~ \text{ } \EE\text{-triangle } Y\to W'\to S\dashrightarrow \text{, }S\in \s \text{ and }W'\in \W \}$,
\item[(iii)] $\h=\B^+\cap \B^-$.
\end{itemize}
We call the ideal quotient $\underline \h$ the heart of $((\s,\T),(\U,\V))$. For a single cotorsion pair $(\U',\V')$, the heart of $((\U',\V'),(\U',\V'))$ is called the heart of the cotorsion pair $(\U',\V')$.
\end{defn}

We will use the following property in the next section.

\begin{lem}\cite[Lemmas 2.9, 2.10 and 2.30]{LN}\label{+-}
\begin{itemize}
\item[(1)] In the $\EE$-triangle $T\to A\xrightarrow{f} B\dashrightarrow$ with $T\in \T$, $A\in \B^+$ whenever $B\in \B^+$.
\item[(2)] In the $\EE$-triangle $T\to A\xrightarrow{f} B\dashrightarrow$ with $T\in \T$ and $A,B\in \h$, $\underline f$ is a monomorphism in $\underline \h$.
\item[(3)] In the $\EE$-triangle $X\xrightarrow{x} Y\to U\dashrightarrow$ with $U\in \U$, $Y\in \B^-$ whenever $X\in \B^-$.
\item[(4)] In the $\EE$-triangle $X\xrightarrow{x} Y\to U\dashrightarrow$ with $U\in \U$  and $X,Y\in \h$, $\underline x$ is an epimorphism in $\underline \h$.
\end{itemize}
\end{lem}

We have already known that: for any twin cotorsion pair $((\s,\T),(\U,\V))$, its heart is a semi-abelian category \cite[Theorem 2.32]{LN}. Moreover, if $\h=\B^+$ or $\h=\B^-$, the heart of $((\s,\T),(\U,\V))$ is quasi-abelian \cite[Theorem 4.2]{HS}. In the next section, we show that the hearts of twin cotorsion pairs are always quasi-abelian.

\section{The hearts are quasi-abelian}

In order to show our main results, we need the following definition.

\begin{defn}\label{DF1}
For any pair of objects $C,A\in \B$, let $\EE_1(C,A)$ be the subset of $\EE(C,A)$ such that $\delta\in \EE_1(C,A)$ if it admits an $\EE$-triangle
$$\xymatrix{A\ar[r]^x &B\ar[r]^y &C\ar@{-->}[r]^{\delta} &}$$
which satisfies the following conditions:
\begin{itemize}
\item[(a)] $x$ is $\W$-monic, which means $\Hom_{\B}(B,W)\xrightarrow{\Hom_{\B}(x,W)}\Hom_{\B}(A,W)$ is surjective for any $W\in \W$.
\item[(b)] $y$ is $\W$-epic, which means $\Hom_{\B}(W,B)\xrightarrow{\Hom_{\B}(W,y)}\Hom_{\B}(W,C)$ is surjective for any $W\in \W$.
\end{itemize}
For convenience, such $\EE$-triangles are called $\EE_1$-triangles.
\end{defn}


Although $\h$ is not always extension closed in $(\B,\EE,\mathfrak{s})$, we have the following lemma.

\begin{lem}\label{exten}
Let $A \xrightarrow{x} B\xrightarrow{y} C\dashrightarrow$ be an $\EE_1$-triangle with $C,A\in \h$. Then $B\in \h$.
\end{lem}

\begin{proof}
Since $A\in \h$, it admits an $\EE$-triangle $A \xrightarrow{w} W\rightarrow S \dashrightarrow$ with $W\in \W$ and $S\in \s$. Then $w$ factors through $x$ since $x$ is $\W$-monic. Hence we have the following commutative diagram by the dual of \cite[Proposition 3.15]{NP}.
$$\xymatrix{
A \ar[r]^x \ar[d]_w &B\ar[r]^y \ar[d] &C \ar@{=}[d] \ar@{-->}[r] &\\
W \ar[r] \ar[d] &W\oplus C \ar[r] \ar[d] &C \ar@{-->}[r] &\\
S \ar@{-->}[d] \ar@{=}[r] &S \ar@{-->}[d]\\
&&
}
$$
By Lemma \ref{+-}, the second column of the diagram implies that $B\in \B^-$. Dually we can prove that $B\in \B^+$. Hence $B\in \B^+\cap\B^-=\h$.
\end{proof}

We introduce a useful lemma.

\begin{lem}\cite[Proposition 1.20]{LN}\label{L4}
Let $A\xrightarrow{x} B\xrightarrow{y} C\overset{\delta}\dashrightarrow$ be any $\EE$-triangle. Let $f:A\to D$ be any morphism and $D \xrightarrow{d} E\xrightarrow{e} C\dashrightarrow$ be an $\EE$-triangle realizing $f_*\delta$. Then there is a morphism $g:B\to E$ which gives a commutative diagram of $\EE$-triangles
$$\xymatrix{
A\ar[r]^x \ar[d]_{f} &B\ar[r]^y \ar[d]^{\exists ~g} &C\ar@{-->}[r]^{\delta} \ar@{=}[d]&\\
D \ar[r]_d &E\ar[r]_e &C\ar@{-->}[r]_{f_*\delta} &
}
$$
and an $\EE$-triangle $A \xrightarrow{\svecv{-f}{x}} D\oplus B \xrightarrow{\svech{d}{g}} E \overset{e^*\delta}\dashrightarrow$. This means the left square of the commutative diagram is a weak push-out and  a weak pull-back:
\begin{itemize}
\item[(i)] If there are morphisms $s:D\to R$ and $t:B\to R$ such that $sf=tx$, then there exists a morphism (not necessarily unique) $r:E\to R$ which makes the following diagram commute:
$$\xymatrix{
A\ar[r]^x \ar[d]_{f} &B \ar[d]^g \ar@/^8pt/[ddr]^t\\
D \ar[r]^d \ar@/_8pt/[drr]_s &E \ar@{.>}[dr]^r\\
&&R
}
$$
\item[(ii)] If there are morphisms $s':R'\to D$ and $t':R'\to B$ such that $ds'=gt'$, then there exists a morphism (not necessarily unique) $r':R'\to A$ which makes the following diagram commute:
$$\xymatrix{
R' \ar@/_8pt/[ddr]_{s'} \ar@/^8pt/[drr]^{t'} \ar@{.>}[dr]^{r'}\\
&A\ar[r]^x \ar[d]_{f} &B \ar[d]^g\\
&D \ar[r]^d  &E
}
$$
\end{itemize}
\end{lem}

Denote by $\mathfrak{s}|_{\EE_1}$ the restriction of $\mathfrak{s}$ onto the $\EE$-extensions realized by $\EE_1$-triangles. We have the following proposition.

\begin{prop}\label{main1}
$(\B,\EE_1,\mathfrak{s}|_{\EE_1})$ is an extriangulated subcategory of $(\B,\EE,\mathfrak{s})$.
\end{prop}

\begin{proof}
By \cite[Proposition 3.16]{HLN} and duality, we need to prove that:
\begin{itemize}
\item[(1)] For any $\delta \in \EE_1(C,A)$ and any morphism $f:A\to D$, we have $f_*\delta\in \EE_1(C,D)$.
\item[(2)] If we have two $\EE_1$-triangles
$$\xymatrix{A \ar[r]^f &B\ar[r]^{f'} &C\ar@{-->}[r]^{\delta} &, &B\ar[r]^g &D\ar[r]^{g'} &F\ar@{-->}[r]^{\rho}&,}$$
then $gf$ admits an $\EE_1$-triangles $A \xrightarrow{gf} D\to E\overset{\sigma}\dashrightarrow$.
\end{itemize}

We check (1) first. Let $\delta$ be realized by an $\EE_1$-triangle $A \xrightarrow{x} B\xrightarrow{y} C\dashrightarrow$ and $f_*\delta$ be realized by an $\EE$-triangle $D \xrightarrow{d} E\xrightarrow{e} C\dashrightarrow$. By Lemma \ref{L4}, there is a morphism $g:B\to E$ which gives a  commutative diagram of $\EE$-triangles
$$\xymatrix{
A\ar[r]^x \ar[d]_{f} &B\ar[r]^y \ar[d]^g &C\ar@{-->}[r]^{\delta} \ar@{=}[d]&\\
D \ar[r]_d &E\ar[r]_e &C\ar@{-->}[r]_{f_*\delta} &
}
$$
such that $A \xrightarrow{\svecv{-f}{x}} D\oplus B \xrightarrow{\svech{d}{g}} E \overset{e^*\delta}\dashrightarrow$ is an $\EE$-triangle.

(1-a) {\bf $d$ is $\W$-monic:} Let $w:D\to W$ be any morphism with $W\in \W$. Since $x$ is $\W$-monic, there is a morphism $w':B\to W$ such that $wf=w'x$. Since the left square of the commutative diagram is a weak push-out, there is a morphism $e_1:E\to W$ such that $e_1d=w$, which implies that $d$ is $\W$-monic.


(1-b) {\bf $e$ is $\W$-epic:} This is just because $y$ is $\W$-epic.

Now we show that (2) holds. For any two $\EE_1$-triangles
$$\xymatrix{A \ar[r]^f &B\ar[r]^{f'} &C\ar@{-->}[r]^{\delta} &, &B\ar[r]^g &D\ar[r]^{g'} &F\ar@{-->}[r]^{\rho}&,}$$
we can get the following commutative diagram (see (ET4) in \cite[Definition 2.12]{NP}).
$$\xymatrix{
A \ar[r]^f \ar@{=}[d] &B\ar[r]^{f'} \ar[d]^g &C \ar[d]^h \ar@{-->}[r]^{\delta} &\\
A \ar[r]^{gf} &D \ar[d]^{g'} \ar[r]^e &E \ar[d]^{h'} \ar@{-->}[r]^{\sigma} &\\
&F \ar@{=}[r] \ar@{-->}[d]^{\rho} &F \ar@{-->}[d]^{f'_*\rho}\\
&&
}
$$
We show that $\sigma\in \EE_1(E,A)$.

(2-a) {\bf $gf$ is $\W$-monic:} This is just because $g$ and $f$ are $\W$-monic.

(2-b) {\bf $e$ is $\W$-epic:} Let $w:W\to E$ be any morphism with $W\in \W$. Since $g'$ is $\W$-epic, there exists a morphism $w_1:W\to D$ such that $h'w=g'w_1=h'ew_1$. Then there is a morphism $w_2:W\to C$ such that $hw_2=w-ew_1$. Since $f'$ is $\W$-epic, there is a morphism $w_3:W\to B$ such that $w_2=f'w_3$. Hence $hw_2=hf'w_3=egw_3$. Then $w=e(w_1+gw_3)$, which means $e$ is $\W$-epic.
\end{proof}

Denote $\EE_1|_{\h}$ by $\mathbb{F}$. For convenience, any $\EE$-triangle which realizes an $\EE$-extension lying in some extension group $\mathbb{F}(C,A)$ is called an $\mathbb{F}$-triangle. Let $\mathfrak{F}$ be the class of all $\mathbb{F}$-triangles. By Remark \ref{rem}, Lemma \ref{exten} and Proposition \ref{main1}, we have the following corollary.

\begin{cor}
$(\h,\mathbb{F},\mathfrak{s}|_{\mathbb{F}})$ is an extriangulated subcategory of $(\B,\EE,\mathfrak{s})$.
\end{cor}

We can get the following property for $\mathbb{F}$-triangles.

\begin{lem}\label{ick}
Let $A\xrightarrow{x} B\xrightarrow{y} C\dashrightarrow$ be an $\mathbb{F}$-triangle. Then $A\xrightarrow{\underline x} B\xrightarrow{\underline y} C$ is a kernel-cokernel pair in $\underline \h$.
\end{lem}

\begin{proof}
We only need to show that $\underline x$ is the kernel of $\underline y$, the other half is by dual.

$C$ admits an $\EE$-triangle $V_C\to W_C\to C\dashrightarrow$ with $V_C\in \V$ and $W_C\in \W$,  since $y$ is $\W$-epic, we can get the following commutative diagram
$$\xymatrix{
V_C \ar[r] \ar[d] &W_C \ar[r] \ar[d]^w &C \ar@{=}[d] \ar@{-->}[r] &\\
A \ar[r]_x &B \ar[r]_y &C \ar@{-->}[r] &.
}
$$
By Lemma \ref{L4}, we can choose $w$ to make an $\EE$-triangle $V_C\longrightarrow A\oplus W_C\xrightarrow{\svecv{x}{w}} B\dashrightarrow$ which implies that $\underline x$ is a monomorphism in $\underline \h$.

Let $z:D\to B$ be a morphism in $\h$ such that $\underline {yz}=0$. Then we have a commutative diagram
$$\xymatrix{
D \ar[d]_{z} \ar[r]^{w_1} &W \ar[d]^{w_2}\\
B \ar[r]_{y} &C
}
$$
with $W\in \W$. Since $y$ is $\W$-epic, there is a morphism $w_3:W\to B$ such that $w_2=yw_3$. Then $y(z-w_3w_1)=0$. There is a morphism $d:D\to A$ such that $xd=z-w_3w_1$. Hence $\underline {xd}=\underline z$, which implies that $\underline x$ is the kernel of $\underline y$.
\end{proof}

According to this lemma, we can get the following corollary.

\begin{cor}\label{serre}
Let $A\xrightarrow{x} B\xrightarrow{y} C\dashrightarrow$ be a nonsplit $\mathbb{F}$-triangle. Then $A,B,C$ do not belong to $\W$.
\end{cor}

\begin{proof}
If $A\in \W$, since $x$ is $\W$-monic, $1:A\to A$ factors through $x$, which means this $\mathbb{F}$-triangle splits, a contradiction. Hence $A\notin \W$. By the similar argument we can obtain that $C\notin \W$.

If $B\in \W$, by Lemma \ref{exten}, $A\to 0\to B$ is a kernel-cokernel pair in $\underline \h$, which implies that $A,C$ becomes zero objects in $\underline \h$. Then $A,C\in \W$, a contradiction.
\end{proof}

By the definition of $\mathfrak{F}$, we can get that $\W$ is a subcategory of projective-injective objects in $(\h,\mathbb{F},\mathfrak{s}|_{\mathbb{F}})$. A well-known result by Happel \cite{Ha} is that the stable category of a Frobenius exact category is a triangulated category. On extriangulated categories, we have the following generalized result:

\begin{prop}\cite[Proposition 3.30]{NP} and \cite[Theorem 3.13]{ZZ1}
Let $\mathcal I\subseteq \B$ be a subcategory of projective-injective objects. Then the ideal quotient $\B/\mathcal I$ has the structure of an extriangulated category, induced from that of $\B$.
\end{prop}

For $(\h,\mathbb{F},\mathfrak{s}|_{\mathbb{F}})$  and $\W$, the ``$\EE$-triangles" in $\underline \h$ are just the image of $\mathbb{F}$-triangles under the canonical quotient functor $\pi:\h\to \underline \h$. We denote the class of these ``$\EE$-triangles" by $\underline {\mathfrak{F}}$. For any sequence $A\xrightarrow{\underline a} B\xrightarrow{\underline b} C$ in $\underline {\mathfrak{F}}$, $\underline a$ is a monomorphism and $\underline b$ is an epimorphism. Then by \cite[Corollary 3.18]{NP}, $\underline \h$ has an exact category structure in the sense of \cite{Bu} and the class of short exact sequences is just $\underline {\mathfrak{F}}$. Thus $(\underline \h,\underline {\mathfrak{F}})$ is an exact category.


\begin{thm}\label{ck}
$(\underline \h,\underline {\mathfrak{F}})$ is the largest exact category structure on $\underline \h$ in the sense that any kernel-cokernel pair in $\underline \h$ belongs to $\underline {\mathfrak{F}}$.
\end{thm}

\begin{proof}
We have known that $(\underline \h,\underline {\mathfrak{F}})$ is an exact category. Let $A\xrightarrow{\underline x} B\xrightarrow{\underline y} C$ be a kernel-cokernel pair in
$\underline{\mathcal H}$. Since $A\in \h$, it admits an $\EE$-triangle $A\to W^A \to S^A \dashrightarrow$
with $S^A\in \s$ and $W^A\in \W$. Then we have the following commutative diagram of $\EE$-triangles
$$\xymatrix{
A\ar[r] \ar[d]_x &W^A \ar[r] \ar[d] &S^A \ar@{=}[d] \ar@{-->}[r] &\\
B\ar[r]_{y_1} &C_1 \ar[r] &S^A \ar@{-->}[r] &
}
$$
with $C_1\in \B^-$ by Lemma \ref{+-}. $C_1$ admits the following commutative diagram
$$\xymatrix{
V_1 \ar[r] \ar@{=}[d] &U_1 \ar[r] \ar[d] &C_1 \ar[d]^{y_2} \ar@{-->}[r] &\\
V_1 \ar[r] &W_1 \ar[r] \ar[d] &C' \ar[d] \ar@{-->}[r] &\\
&S_1 \ar@{-->}[d] \ar@{=}[r] &S_1 \ar@{-->}[d]\\
&&
}
$$
with $U_1\in \U$, $S_1\in \s$, $V_1\in \V$, $W_1\in \W$ and $C'\in \h$. By the dual of \cite[Proposition 2.25]{LN}, $\underline {y_2y_1}$ is the cokernel of $\underline x$. Denote $y_2y_1$ by $y'$, we have the following commutative diagram
$$\xymatrix{
B \ar[r]^{y_1} \ar@{=}[d] &C_1 \ar[r] \ar[d]^{y_2} &S^A \ar[d] \ar@{-->}[r] &\\
B \ar[r]_{y'} &C' \ar[r] \ar[d] &S' \ar[d] \ar@{-->}[r] &\\
&S_1 \ar@{=}[r] \ar@{-->}[d]  &S_1 \ar@{-->}[d]\\
&&
}
$$
with $S'\in \s$. Since $C'$ admits an $\EE$-triangle $V_1\to W_1\xrightarrow{w_2} C'\dashrightarrow$ with $V_1\in \V$ and $W_1\in \W$, we have the following commutative diagram
$$\xymatrix{
V_1 \ar@{=}[r] \ar[d] &V_1 \ar[d]\\
A' \ar[r]^{w_1} \ar[d]_{x'} &W_1 \ar[r]^{s_1} \ar[d]^{w_2} &S' \ar@{=}[d] \ar@{-->}[r]^{\delta} &\\
B \ar[r]_{y'} \ar@{-->}[d]  &C' \ar[r]_{s_2} \ar@{-->}[d] &S'\ar@{-->}[r]^{x'_*\delta} &\\
&&
}
$$
with $A'\in \h$ by Lemma \ref{+-}. By Lemma \ref{L4}, there is a morphism $w_2':W_1\to C'$ and a commutative diagram of $\EE$-triangles
$$\xymatrix{
A' \ar[r]^{w_1} \ar[d]_{x'} &W_1 \ar[r]^{s_1} \ar[d]^{w_2'} &S' \ar@{=}[d] \ar@{-->}[r]^{\delta} &\\
B \ar[r]_{y'} &C' \ar[r]_{s_2} &S'\ar@{-->}[r]^{x'_*\delta} &
}
$$
which induces an $\EE$-triangle
$$\xymatrix{A' \ar[r]^-{\svecv{w_1}{-x'}} &W_1\oplus B \ar[r]^-{\svech{w_2'}{y'}} &C' \ar@{-->}[r] &.}$$
By Proposition \ref{exs}, we have the following exact sequence
$$\Hom_\B(W_1,\W)\xrightarrow{\Hom_\B(w_1,\W)}\Hom_\B(A',\W)\longrightarrow \EE(S',\W)=0$$
which implies that $w_1$ is $\W$-monic. Then $\svecv{w_1}{-x'}$ is also $\W$-monic. Now we show $\svech{w_2'}{y'}$ is $\W$-epic. Since $s_2w_2'=s_1=s_2w_2$, there is a morphism $w:W_1\to B$ such that $w_2-w_2'=y'w$. For any morphism $w':W'\to C'$ with $W'\in \W$, since $\EE(W',V_1)=0$, there is a morphism $u:W'\to W_1$ such that $w_2u=w'$. Then $w'=y'wu+w_2'u$ and we have the following commutative diagram
$$\xymatrix{
&&W' \ar[dll]_-{\svech{u}{wu}} \ar[d]^{w'}\\
W_1\oplus B \ar[rr]_-{\svech{w_2'}{y'}} &&C'.
}
$$
Hence $\svech{w_2'}{y'}$ is $\W$-epic. By Lemma \ref{+-} (2) and (4), $-\underline {x'}$ is a monomorphism and $\underline {y'}$ is an epimorphism. Then $A'\xrightarrow{-\underline {x'}} B\xrightarrow{\underline {y'}}C'$ belongs to $\underline {\mathfrak{F}}$.
$A'\xrightarrow{-\underline {x'}} B\xrightarrow{\underline {y'}}C'$ is a kernel-cokernel pair by Lemma \ref{ick} and it is isomorphic to $A\xrightarrow{\underline x} B\xrightarrow{\underline y} C$. Since $(\underline \h,\underline {\mathfrak{F}})$ is an exact category, $\underline {\mathfrak{F}}$ is closed under isomorphisms. Then $A\xrightarrow{\underline x} B\xrightarrow{\underline y} C$ belongs to $\underline {\mathfrak{F}}$.
\end{proof}

\begin{rem}
By Lemma \ref{ick} and Theorem \ref{ck}, we know that a sequence $A\xrightarrow{\underline x} B\xrightarrow{\underline y} C$ is a kernel-cokernel pair if and only if it is the image of an $\mathbb{F}$-triangle by canonical quotient functor $\pi:\B\to \uB$.
\end{rem}

The following lemma is needed in the proofs of the main results.

\begin{lem}\label{kc}
Let $\mathcal A$ be a preabelian category.
\begin{itemize}
\item[(1)] If $k:K\to A$ is the kernel of a morphism $f:A\to B$ and $c:A\to C$ is the cokernel of $k$, then $K\xrightarrow{k} A\xrightarrow{c} C$ is a kernel-cokernel pair.
\item[(2)] If $k:K\to A$ is both an epimorphism and a kernel, then it is an isomorphism.
\end{itemize}
\end{lem}

\begin{proof}
(1) We only need to show that $k$ is the kernel of $c$. First there exists a morphism $b:C\to B$ such that $f=bc$. Let $g:D\to A$ be a morphism such that $cg=0$, then $bcg=fg=0$. Hence there exits a morphism $d:D\to K$ such that $dg=k$, which implies that $k$ is the kernel of $c$.\\
(2) By the result we just proved, $k$ admits a kernel-cokernel pair $K\xrightarrow{k} A\xrightarrow{c} C$. Since $k$ is an epimorphism, we get that $c=0$. $0:A\to C$ has a kernel $1:A\to A$, hence we have a commutative diagram
$$\xymatrix{
K \ar[r]^k \ar[d]_{\cong} &A \ar[r]^0 &C\\
A \ar@{=}[ur]
}$$
which implies that $k$ is an isomorphism.
\end{proof}

\begin{rem}\label{cond}
There are several equivalent conditions for a preabelian category $\mathcal A$ being abelian. One is that any epic-monic morphism is an isomorphism. By the dual of Lemma \ref{kc} (2), we can get the following condition for $\mathcal A$ being abelian: any epimorphism in $\mathcal A$ is a cokernel. Note that in an abelian category, epimorphisms are always cokernels.
\end{rem}

Now we can show the following theorem.

\begin{thm}\label{main}
$\underline \h$ is quasi-abelian.
\end{thm}

\begin{proof}
Let $\underline i:A\to B$ be a kernel in $\underline \h$. Since $\underline \h$ is preabelian, $\underline i$ admits a kernel-cokernel pair by Lemma \ref{kc} (1), which also means that it admits a short exact sequence
$$\xymatrix{A \ar[r]^{\underline i} &B \ar[r]^{\underline p} &C}$$
in $(\underline \h,\underline {\mathfrak{F}})$. For any push-out diagram
$$\xymatrix{
A \ar[r]^{\underline i} \ar[d] &B \ar[d]\\
A' \ar[r]_{\underline {i'}} &B'
}
$$
by the definition of exact category, $\underline {i'}$ admits a short exact sequence $A' \xrightarrow{\underline {i'}} B'\xrightarrow{\underline {p'}} C'$. Hence $\underline {i'}$ is a kernel, which implies that $\underline \h$ is right quasi-abelian. By the similar method we can show that $\underline \h$ is left quasi-abelian. Thus $\underline \h$ is quasi-abelian.
\end{proof}

\section{Abelian Hearts}

The hearts of twin cotorsion pairs are not always abelian (see Example \ref{ex1}). We will give a necessary condition for the hearts being abelian first, after some preparations.

Since $(\s,\T)$ and $(\U,\V)$ are cotorsion pairs, they have abelian hearts by \cite[Theorem 3.2]{LN}. For convenience, we introduce the following notations:
\begin{itemize}
\item[(a1)] $\B^+_1=\{X\in \B \text{ }|\text{ } \mbox{there exists an}~\text{ } \EE\text{-triangle } T\to M\to X\dashrightarrow \text{, }T\in \T \text{ and }M\in \s\cap\T \}$;
\item[(a2)] $\B^-_1=\{Y\in \B \text{ }|\text{ } \mbox{there exists an}~ \text{ } \EE\text{-triangle } Y\to M'\to S\dashrightarrow \text{, }S\in \s \text{ and }M'\in \s\cap\T \}$;
\item[(a3)] $\h_1=\B^+_1\cap \B^-_1$;
\item[(b1)] $\B^+_2=\{X\in \B \text{ }|\text{ } \mbox{there exists an}~\text{ } \EE\text{-triangle } V\to N\to X\dashrightarrow \text{, }V\in \V \text{ and }N\in \U\cap\V \}$;
\item[(b2)] $\B^-_2=\{Y\in \B \text{ }|\text{ } \mbox{there exists an}~ \text{ } \EE\text{-triangle } Y\to N'\to U\dashrightarrow \text{, }U\in \U \text{ and }N'\in \U\cap\V \}$;
\item[(b3)] $\h_2=\B^+_2\cap \B^-_2$.
\end{itemize}
By definition $\h_1/(\s\cap \T)$ is the heart of $(\s,\T)$, and $\h_2/(\U\cap \V)$ is the heart of $(\U,\V)$. We also have $\B^-_1\subseteq \B^-$ and $\B^+_2\subseteq \B^+$.

\begin{lem}\label{L41}
$\B^+\subseteq \B^+_1$ and $\B^-\subseteq \B^-_2$.
\end{lem}

\proof For an object $X\in \B^{+}$,
there exists an $\EE$-triangle
$V\to W\to X\dashrightarrow$
where $V\in\V$ and $W\in\W$.
Since $(\s,\T)$ is a cotorsion pair,
there exists an  $\EE$-triangle
$T\to S\to X\dashrightarrow$
where $T\in\T$ and $S\in\s$. We can obtain the following commutative diagram
$$\xymatrix{
&T\ar@{=}[r] \ar[d] &T \ar[d]\\
V\ar[r]\ar[d]&M\ar[r] \ar[d] &S \ar[d] \ar@{-->}[r] &\\
V\ar[r]&W \ar[r] \ar@{-->}[d] &X \ar@{-->}[r] \ar@{-->}[d]&\\
&&
}$$
with $M\in\T$. Since $\EE(S,V)=0$, the second row of the above diagram splits.
Thus we have $M\cong S\oplus V$.
Note that $\T$ is closed under direct summands,
we have $S\in\T$.
Hence $S\in\s\cap\T$, which implies that $X\in\B^+_1$.
Dually, one can show that $\B^-\subseteq \B^-_2$.
\qed

\begin{lem}\label{FT}
Let $f:A\to B$ be a morphism in $\B$.
\begin{itemize}
\item[(1)] Assume that $A,B\in \h_1$, then $f$ factors through $\W$ if and only if it factors through $\s\cap \T$.
\item[(2)] Assume that $A,B\in \h_2$, then $f$ factors through $\W$ if and only if it factors through $\U\cap \V$.
\end{itemize}
\end{lem}

\begin{proof}
We only show (1), the proof of (2) is by dual. Since $((\s,\T),(\U,\V))$ is a twin cotorsion pair,
we have $\s\subseteq\U$ which implies $$\s\cap\T\subseteq\U\cap\T=\W.$$
This shows that if $f$ factors through $\s\cap\T$, then it factors through
$\W$.

Conversely, assume that $f$ factors through $\W$.
Then there exist two morphisms $s\colon A\to W$ and
$t\colon W\to B$ such that $f=ts$ where $W\in\W$.

Since $A\in\h_1$, there exists an
$\EE$-triangle
$$A\xrightarrow{~g~}M\to S\dashrightarrow$$
where $M\in\s\cap\T$ and $S\in\s$. Since $\EE(S,W)=0$, there exists a morphism
$b\colon M\to W$ such that $s=bg$ and then
$f=ts=(tb)g$.
This shows that $f$ factors through $M\in \s\cap\T$.
\end{proof}

By this lemma, we have $\h_1/(\s\cap \T)=\underline {\h_1}$ and $\h_2/(\U\cap \V)=\underline {\h_2}$. Now we can state the necessary condition.

\begin{prop}\label{prop1}
If $\underline \h$ is abelian, then $\underline {\h_1}$ and $\underline {\h_2}$ are dense in $\underline \h$. Moreover, if $\B$ is Krull-Schmidt, $\underline \h=\underline {\h_1}\cap \underline {\h_2}$.
\end{prop}

\begin{proof}
Let $X\in \h$. Then we have the following commutative diagram
$$\xymatrix{
T_1 \ar@{=}[r] \ar[d] &T_1 \ar[d]\\
X_1 \ar[r]^m \ar[d]_{f} &M_1 \ar[r] \ar[d] &S_1 \ar@{=}[d] \ar@{-->}[r] &\\
X \ar[r]_w \ar@{-->}[d] &W_1 \ar[r] \ar@{-->}[d] &S_1 \ar@{-->}[r] &\\
&&
}
$$
with $W_1\in \W$, $S_1\in \s$, $M_1\in \s\cap\T$ and $T_1\in \T$. By Lemma \ref{+-}, the first column of the diagram implies that $X_1\in \B^+$, the second row implies that $X_1\in \B^-_1$. Since $\B^+\subseteq \B^+_1$ by Lemma \ref{L41} and $\B^-_1\subseteq \B^-$, we have $X_1\in \h_1\cap \h$. From the first column we can also get that $\underline f$ is a monomorphism in $\underline \h$. The diagram induces an $\EE$-triangle $$X_1\xrightarrow{\svecv{f}{m}} X\oplus M_1 \longrightarrow W_1\dashrightarrow$$ which implies that $\underline f$ is an epimorphism in $\underline \h$ by Lemma \ref{+-}. Since $\underline \h$ is abelian, $\underline f$ is an isomorphism. Dually we can show that $X$ is isomorphic to an object in $\h_2\cap \h$ in the quotient category $\underline \h$. Hence $\underline {\h_1}$ and $\underline {\h_2}$ are dense in $\underline \h$.

Now assume that $\B$ is Krull-Schmidt and $X$ is indecomposable which does not lie in $\W$. Since $\B^-_1\subseteq \B^-$ and $\B^+_2\subseteq \B^+$, we have
$$\h_1\cap \h_2=(\B^-_1\cap\B^+_1)\cap(\B^-_2\cap\B^+_2)\subseteq\B^-\cap\B^+=\h.$$
Since $\underline f$ is an isomorphism, there exists a morphism $g:X\to X_1$ such that $\underline {fg}=\underline 1$. Then $1-fg$ factors through an object $W\in \W$. Assume $1-fg:X\xrightarrow{w_0} W\xrightarrow{w_0'} X$. Since $\B$ is Krull-Schmidt and $X$ is indecomposable, $\End_{\B}(X)$ is a local ring. $X$ does not lie in $\W$ implies that $w_0'w_0$ is not invertible. Hence $fg=1-w_0'w_0$ is invertible. Let $(fg)^{-1}=h$. Since $\EE(S_1,W_1)=0$ and $\EE(S_1,M_1)=0$, we have the following commutative diagram
$$\xymatrix{
X \ar[r]^{w} \ar[d]_{gh^{-1}}  &W_1 \ar[d]^{w_1} \ar[r] &S_1 \ar[d] \ar@{-->}[r] &\\
X_1 \ar[r]^m \ar[d]_{f} &M_1 \ar[r] \ar[d]^{m_1} &S_1 \ar[d] \ar@{-->}[r] &\\
X \ar[r]_{w}  &W_1 \ar[r] &S_1 \ar@{-->}[r] &
}
$$
We can get that $1-m_1w_1$ factors through $S_1$. Hence $W_1$ is a direct summand of $M_1\oplus S_1$, which implies that $W_1\in \s\cap \T$. Thus $X\in \B^-_1\cap\B^+\subseteq \h_1$. Dually we can show that $X\in \h_2$. Hence $X\in \h_1\cap \h_2$, we can obtain that $\underline \h=\underline {\h_1}\cap \underline {\h_2}$.
\end{proof}

The following example shows that not every heart of a cotorison pair is abelian.

\begin{exm}\label{ex1}
	{\upshape Let $A=kQ/I$ be a self-injective algebra given by the following quiver $$Q:
		\setlength{\unitlength}{0.03in}\xymatrix{1 \ar@<0.5ex>[r]^{\alpha}_{\ } & 2\ar@<0.5ex>[l]^{\beta}}$$ with $I=\langle\alpha\beta\alpha\beta, \beta\alpha\beta\alpha\rangle$.
		The Auslander-Reiten quiver of $\mod A$ is
		\begin{center}
			$
			\xymatrix@!@C=0.1mm@R=0.1mm{
				&\txt{1\\2\\1\\2}\ar[dr]&&\txt{2\\1\\2\\1}\ar[dr]&\\
				\txt{2\\1\\2}\ar[ur]\ar[dr]\ar@{--}[u]&&\txt{1\\2\\1}\ar@{.}[ll]\ar[ur]\ar[dr]&&\txt{2\\1\\2}\ar@{.}[ll]\ar@{--}[u]\\
				&\txt{2\\1}\ar[dr]\ar[ur]\ar@{.}[l]&&\txt{1\\2}\ar[dr]\ar[ur]\ar@{.}[ll]&\ar@{.}[l]\\
				1\ar[ur]\ar@{--}[uu]&&{2}\ar@{.}[ll]\ar[ur]&&1\ar@{.}[ll]\ar@{--}[uu] }
			$
			\medskip
			
			Figure 1
\end{center}
where the first and the last columns are identified. The stable module category $\B:={\rm \underline{\mod}}A$ is a
 triangulated category and we can get the Auslander-Reiten quiver of $\B$ by deleting the first row in Figure 1.}
\end{exm}

Let  {\setlength{\jot}{-3pt}
$$(\s, \T):=(\add\bigg\{\begin{aligned}2\\1\\2\end{aligned}~\bigg\}, \add\bigg\{\begin{aligned}2\\1\\2\end{aligned}\oplus 2\oplus \begin{aligned}1\\2\end{aligned}~\bigg\}),\quad (\U, \V):=(\T, \add\{2\}).$$
	
By definition, they form a twin cotorsion pair, and
	$$\underline {\h_1}=\underline {\h_2}=\add\bigg\{1\bigg\};\quad \underline \h=\add\bigg\{1\oplus \begin{aligned}2\\1\end{aligned}\oplus\begin{aligned}1\\2\\1\end{aligned}\bigg\}$$
By Proposition \ref{prop1}, $\underline \h$ is not abelian.
}

\medskip


By using the results in Section 3, we provide a sufficient and necessary condition for the hearts of twin cotorsion pairs being abelian.

\begin{thm}\label{main2}
$\underline \h$ is an abelian category if and only if the following condition is satisfied:
\begin{itemize}
\item Any epimorphism $\underline p:B\to C$ in $\underline \h$ admits an $\EE$-triangle
$$\xymatrix{B' \ar[r]^-{p'} &C' \ar[r] &S \ar@{-->}[r] &}$$
such that $B',C'\in \h$, $S\in \s$ and $\underline {p'}=\underline p$.
\end{itemize}
\end{thm}

\begin{proof}

By Remark \ref{cond}, we only need to show the following fact:
\begin{itemize}
\item A morphism $\underline p:B\to C$ in $\underline \h$ is a cokernel if and only if we have an $\EE$-triangle
$$\xymatrix{B' \ar[r]^-{p'} &C'\ar[r] &S \ar@{-->}[r] &}$$
such that $B',C'\in \h$, $S\in \s$ and $\underline {p'}=\underline p$.
\end{itemize}
We show the ``only if" part first.

By Lemma \ref{kc} (1) and Theorem \ref{ck}, $\underline p$ admits an $\mathbb{F}$-triangle
$$\xymatrix{A' \ar[r]^{i'} &B' \ar[r]^{p_1} &C_1 \ar@{-->}[r] &}$$
with $\underline {p_1}=\underline p$. Since $i'$ is $\W$-monic and $A'$ admits an $\EE$-triangle $A'\to W\to S\dashrightarrow$ where $W\in \W$, $S\in \s$, we have the following commutative diagram
$$\xymatrix{
A' \ar[r]^{i'} \ar@{=}[d] &B' \ar[r]^{p_1} \ar[d]^{b'} &C_1 \ar[d] \ar@{-->}[r] &\\
A' \ar[r] &W \ar[r] &S \ar@{-->}[r] &.
}
$$
By the dual of Lemma \ref{L4}, we can choose $b'$ to make an $\EE$-triangle
$$\xymatrix{B' \ar[r]^-{\svecv{p_1}{b'}} &C_1\oplus W \ar[r] &S \ar@{-->}[r] &.}$$
This $\EE$-triangle is just what we need.

Now we show the ``if" part. Since $C'$ admits an $\EE$-triangle $V'\to W'\to C'\dashrightarrow$ with $V'\in \V$ and $W'\in \W$, we have the following commutative diagram
$$\xymatrix{
V' \ar@{=}[r] \ar[d] &V' \ar[d]\\
A' \ar[r]^{w'} \ar[d]_{i'} &W' \ar[r] \ar[d] &S \ar@{=}[d] \ar@{-->}[r] &\\
B' \ar[r]_-{p'} \ar@{-->}[d]  &C' \ar[r] \ar@{-->}[d] &S\ar@{-->}[r] &\\
&&
}
$$
with $A'\in \h$. By the proof of Theorem \ref{ck}, we can get an $\mathbb{F}$-triangle
$$\xymatrix{A'\ar[r]^-{\svecv{-i'}{w'}} &B'\oplus W' \ar[r]^-{\svech{p'}{w''}} &C' \ar@{-->}[r] &}$$
which induces a short exact sequence $A'\xrightarrow{-\underline {i'}} B' \xrightarrow {\underline {p'}} C'$. Hence $\underline p=\underline {p'}$ is a cokernel.
\end{proof}

By duality, we obtain the following corollary immediately.

\begin{cor}
$\underline \h$ is an abelian category if and only if the following condition is satisfied:
\begin{itemize}
\item Any monomorphism $\underline q:A\to B$ in $\underline \h$ admits an $\EE$-triangle
$$\xymatrix{V \ar[r] &A'\ar[r]^-{q'} &B' \ar@{-->}[r] &}$$
such that $A',B'\in \h$, $V\in \V$ and $\underline {q'}=\underline q$.
\end{itemize}
\end{cor}

The following example shows that the condition in Proposition \ref{prop1} is only necessary, not sufficient.

\begin{exm}
{\upshape 	 Let $A=kQ/I$ be an algebra given by the quiver
	 \begin{align}
	 	\begin{minipage}{0.6\hsize}
	 		\ \ \ \ \  \xymatrix{Q: \begin{smallmatrix}6\end{smallmatrix}\ar[r]^{\alpha}&\begin{smallmatrix}5\end{smallmatrix}\ar[r]^{\beta}
	 			&\begin{smallmatrix}4\end{smallmatrix}\ar[r]^{\gamma}&\begin{smallmatrix}3\end{smallmatrix}\ar[r]^{\delta}&\begin{smallmatrix}2\end{smallmatrix}\ar[r]^{\varepsilon}&\begin{smallmatrix}1\end{smallmatrix}}\notag
	 	\end{minipage}
	 \end{align}
and $I=\langle\alpha\beta\gamma\delta, \beta\gamma\delta\varepsilon\rangle.$	The Auslander-Reiten quiver of $\mod A$ is the following:}
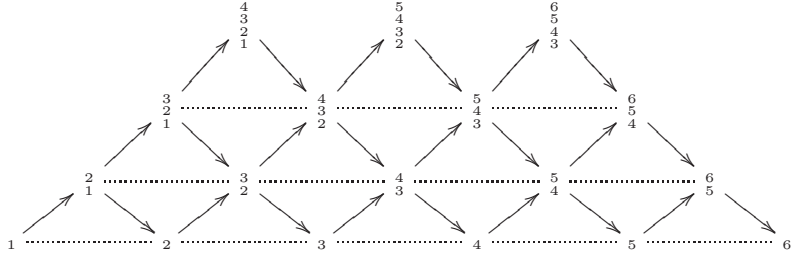
\begin{align}
	\tiny{\xymatrix @R=4mm @C=6mm{&&&
			{\begin{smallmatrix}4\\3\\2\\1\end{smallmatrix}}\ar[rd]&&
			{\begin{smallmatrix}5\\4\\3\\2\end{smallmatrix}}\ar[rd]&&
			{\begin{smallmatrix}6\\5\\4\\3\end{smallmatrix}}\ar[rd]&&&\\
			&&{\begin{smallmatrix}3\\2\\1\end{smallmatrix}}\ar[ru]\ar[rd]&&
			{\begin{smallmatrix}4\\3\\2\end{smallmatrix}}\ar[ru]\ar[rd]\ar@{.}[ll]&&
			{\begin{smallmatrix}5\\4\\3\end{smallmatrix}}\ar[ru]\ar[rd]\ar@{.}[ll]&&
			{\begin{smallmatrix}6\\5\\4\end{smallmatrix}}\ar[rd]\ar@{.}[ll]&&\\
			&{\begin{smallmatrix}2\\1\end{smallmatrix}}\ar[ru]\ar[rd]&&
			{\begin{smallmatrix}3\\2\end{smallmatrix}}\ar[ru]\ar[rd]\ar@{.}[ll]&&
			{\begin{smallmatrix}4\\3\end{smallmatrix}}\ar[ru]\ar[rd]\ar@{.}[ll]&&
			{\begin{smallmatrix}5\\4\end{smallmatrix}}\ar[ru]\ar[rd]\ar@{.}[ll]&&
			{\begin{smallmatrix}6\\5\end{smallmatrix}}\ar[rd]\ar@{.}[ll]\\
			{\begin{smallmatrix}1\end{smallmatrix}}\ar[ru]&&\textnormal{$\begin{smallmatrix}2\end{smallmatrix}$}\ar[ru]\ar@{.}[ll]&&
			{\begin{smallmatrix}3\end{smallmatrix}}\ar[ru]\ar@{.}[ll]&&\textnormal{$\begin{smallmatrix}4\end{smallmatrix}$}\ar[ru]\ar@{.}[ll]&&
			{\begin{smallmatrix}5\end{smallmatrix}}\ar[ru]\ar@{.}[ll]&&
			{\begin{smallmatrix}6\end{smallmatrix}}\ar@{.}[ll]
	}}\notag
\end{align}
\end{exm}

In this example, we denote by ``$\bullet$" in the quiver the objects belong to a subcategory and
by  ``$\circ $" the objects do not. Given four subcategories as follows:
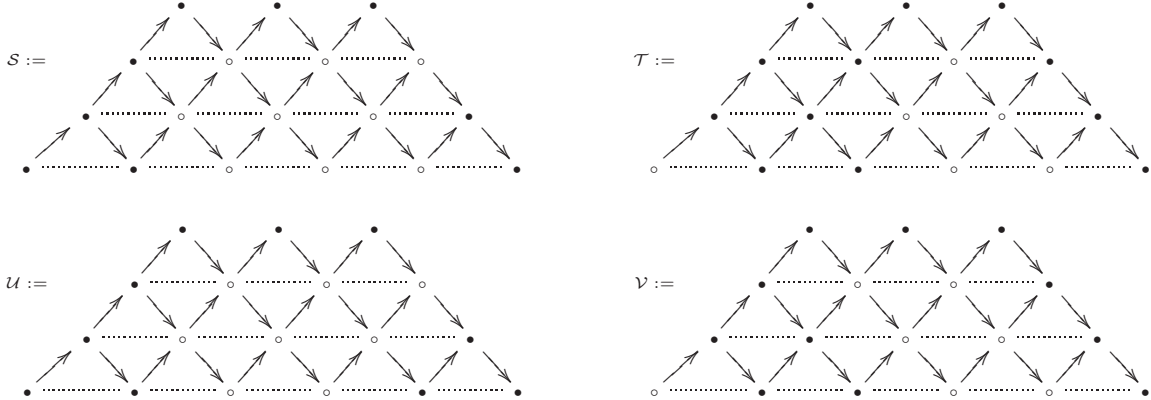
\begin{align}
	\tiny{\xymatrix @R=4mm @C2mm{&&&
			{\begin{smallmatrix}\bullet\end{smallmatrix}}\ar[rd]&&
			{\begin{smallmatrix}\bullet\end{smallmatrix}}\ar[rd]&&
			{\begin{smallmatrix}\bullet\end{smallmatrix}}\ar[rd]&&&\\
			\s:=&&{\begin{smallmatrix}\bullet\end{smallmatrix}}\ar[ru]\ar[rd]&&
			{\begin{smallmatrix}\circ\end{smallmatrix}}\ar[ru]\ar[rd]\ar@{.}[ll]&&
			{\begin{smallmatrix}\circ\end{smallmatrix}}\ar[ru]\ar[rd]\ar@{.}[ll]&&
			{\begin{smallmatrix}\circ\end{smallmatrix}}\ar[rd]\ar@{.}[ll]&&\\
			&{\begin{smallmatrix}\bullet\end{smallmatrix}}\ar[ru]\ar[rd]&&
			{\begin{smallmatrix}\circ\end{smallmatrix}}\ar[ru]\ar[rd]\ar@{.}[ll]&&
			{\begin{smallmatrix}\circ\end{smallmatrix}}\ar[ru]\ar[rd]\ar@{.}[ll]&&
			{\begin{smallmatrix}\circ\end{smallmatrix}}\ar[ru]\ar[rd]\ar@{.}[ll]&&
			{\begin{smallmatrix}\bullet\end{smallmatrix}}\ar[rd]\ar@{.}[ll]\\
			{\begin{smallmatrix}\bullet\end{smallmatrix}}\ar[ru]&&\textnormal{$\begin{smallmatrix}\bullet\end{smallmatrix}$}\ar[ru]\ar@{.}[ll]&&
			{\begin{smallmatrix}\circ\end{smallmatrix}}\ar[ru]\ar@{.}[ll]&&\textnormal{$\begin{smallmatrix}\circ\end{smallmatrix}$}\ar[ru]\ar@{.}[ll]&&
			{\begin{smallmatrix}\circ\end{smallmatrix}}\ar[ru]\ar@{.}[ll]&&
			{\begin{smallmatrix}\bullet\end{smallmatrix}}\ar@{.}[ll]\\
	}}\notag
\ \ \quad\quad \quad\quad	\tiny{\xymatrix @R=4mm @C2mm{&&&
			{\begin{smallmatrix}\bullet\end{smallmatrix}}\ar[rd]&&
			{\begin{smallmatrix}\bullet\end{smallmatrix}}\ar[rd]&&
			{\begin{smallmatrix}\bullet\end{smallmatrix}}\ar[rd]&&&\\
			\T:=&&{\begin{smallmatrix}\bullet\end{smallmatrix}}\ar[ru]\ar[rd]&&
			{\begin{smallmatrix}\bullet\end{smallmatrix}}\ar[ru]\ar[rd]\ar@{.}[ll]&&
			{\begin{smallmatrix}\circ\end{smallmatrix}}\ar[ru]\ar[rd]\ar@{.}[ll]&&
			{\begin{smallmatrix}\bullet\end{smallmatrix}}\ar[rd]\ar@{.}[ll]&&\\
			&{\begin{smallmatrix}\bullet\end{smallmatrix}}\ar[ru]\ar[rd]&&
			{\begin{smallmatrix}\bullet\end{smallmatrix}}\ar[ru]\ar[rd]\ar@{.}[ll]&&
			{\begin{smallmatrix}\circ\end{smallmatrix}}\ar[ru]\ar[rd]\ar@{.}[ll]&&
			{\begin{smallmatrix}\circ\end{smallmatrix}}\ar[ru]\ar[rd]\ar@{.}[ll]&&
			{\begin{smallmatrix}\bullet\end{smallmatrix}}\ar[rd]\ar@{.}[ll]\\
			{\begin{smallmatrix}\circ\end{smallmatrix}}\ar[ru]&&\textnormal{$\begin{smallmatrix}\bullet\end{smallmatrix}$}\ar[ru]\ar@{.}[ll]&&
			{\begin{smallmatrix}\bullet\end{smallmatrix}}\ar[ru]\ar@{.}[ll]&&\textnormal{$\begin{smallmatrix}\circ\end{smallmatrix}$}\ar[ru]\ar@{.}[ll]&&
			{\begin{smallmatrix}\circ\end{smallmatrix}}\ar[ru]\ar@{.}[ll]&&
			{\begin{smallmatrix}\bullet\end{smallmatrix}}\ar@{.}[ll]\\
	}}\notag
\end{align}
\begin{align}
	\tiny{\xymatrix @R=4mm @C2mm{&&&
			{\begin{smallmatrix}\bullet\end{smallmatrix}}\ar[rd]&&
			{\begin{smallmatrix}\bullet\end{smallmatrix}}\ar[rd]&&
			{\begin{smallmatrix}\bullet\end{smallmatrix}}\ar[rd]&&&\\
			\U:=&&{\begin{smallmatrix}\bullet\end{smallmatrix}}\ar[ru]\ar[rd]&&
			{\begin{smallmatrix}\circ\end{smallmatrix}}\ar[ru]\ar[rd]\ar@{.}[ll]&&
			{\begin{smallmatrix}\circ\end{smallmatrix}}\ar[ru]\ar[rd]\ar@{.}[ll]&&
			{\begin{smallmatrix}\circ\end{smallmatrix}}\ar[rd]\ar@{.}[ll]&&\\
			&{\begin{smallmatrix}\bullet\end{smallmatrix}}\ar[ru]\ar[rd]&&
			{\begin{smallmatrix}\circ\end{smallmatrix}}\ar[ru]\ar[rd]\ar@{.}[ll]&&
			{\begin{smallmatrix}\circ\end{smallmatrix}}\ar[ru]\ar[rd]\ar@{.}[ll]&&
			{\begin{smallmatrix}\circ\end{smallmatrix}}\ar[ru]\ar[rd]\ar@{.}[ll]&&
			{\begin{smallmatrix}\bullet\end{smallmatrix}}\ar[rd]\ar@{.}[ll]\\
			{\begin{smallmatrix}\bullet\end{smallmatrix}}\ar[ru]&&\textnormal{$\begin{smallmatrix}\bullet\end{smallmatrix}$}\ar[ru]\ar@{.}[ll]&&
			{\begin{smallmatrix}\circ\end{smallmatrix}}\ar[ru]\ar@{.}[ll]&&\textnormal{$\begin{smallmatrix}\circ\end{smallmatrix}$}\ar[ru]\ar@{.}[ll]&&
			{\begin{smallmatrix}\bullet\end{smallmatrix}}\ar[ru]\ar@{.}[ll]&&
			{\begin{smallmatrix}\bullet\end{smallmatrix}}\ar@{.}[ll]\\
	}}\notag
	\ \ \quad\quad \quad\quad	\tiny{\xymatrix @R=4mm @C2mm{&&&
		{\begin{smallmatrix}\bullet\end{smallmatrix}}\ar[rd]&&
			{\begin{smallmatrix}\bullet\end{smallmatrix}}\ar[rd]&&
			{\begin{smallmatrix}\bullet\end{smallmatrix}}\ar[rd]&&&\\
			\V:=&&{\begin{smallmatrix}\bullet\end{smallmatrix}}\ar[ru]\ar[rd]&&
			{\begin{smallmatrix}\circ\end{smallmatrix}}\ar[ru]\ar[rd]\ar@{.}[ll]&&
			{\begin{smallmatrix}\circ\end{smallmatrix}}\ar[ru]\ar[rd]\ar@{.}[ll]&&
			{\begin{smallmatrix}\bullet\end{smallmatrix}}\ar[rd]\ar@{.}[ll]&&\\
			&{\begin{smallmatrix}\bullet\end{smallmatrix}}\ar[ru]\ar[rd]&&
			{\begin{smallmatrix}\bullet\end{smallmatrix}}\ar[ru]\ar[rd]\ar@{.}[ll]&&
			{\begin{smallmatrix}\circ\end{smallmatrix}}\ar[ru]\ar[rd]\ar@{.}[ll]&&
			{\begin{smallmatrix}\circ\end{smallmatrix}}\ar[ru]\ar[rd]\ar@{.}[ll]&&
			{\begin{smallmatrix}\bullet\end{smallmatrix}}\ar[rd]\ar@{.}[ll]\\
			{\begin{smallmatrix}\circ\end{smallmatrix}}\ar[ru]&&\textnormal{$\begin{smallmatrix}\bullet\end{smallmatrix}$}\ar[ru]\ar@{.}[ll]&&
			{\begin{smallmatrix}\bullet\end{smallmatrix}}\ar[ru]\ar@{.}[ll]&&\textnormal{$\begin{smallmatrix}\circ\end{smallmatrix}$}\ar[ru]\ar@{.}[ll]&&
			{\begin{smallmatrix}\circ\end{smallmatrix}}\ar[ru]\ar@{.}[ll]&&
			{\begin{smallmatrix}\bullet\end{smallmatrix}}\ar@{.}[ll]\\
	}}\notag
\end{align}
then $(\s, \T), (\U, \V)$ form a twin cotorsion pair and
$$\underline {\h_1}=\add\bigg\{\begin{smallmatrix}5\end{smallmatrix}\oplus \begin{smallmatrix}4\\3\end{smallmatrix}\oplus\begin{smallmatrix}5\\4\\3\end{smallmatrix}\bigg\},\quad \underline {\h_2}=\add\bigg\{\begin{smallmatrix}4\\3\\2\end{smallmatrix}\oplus \begin{smallmatrix}4\\3\end{smallmatrix}\oplus\begin{smallmatrix}5\\4\\3\end{smallmatrix}\bigg\},\quad \underline \h=\add\bigg\{\begin{smallmatrix}4\\3\end{smallmatrix}\oplus\begin{smallmatrix}5\\4\\3\end{smallmatrix}\bigg\}=\underline {\h_1}\cap\underline {\h_2}.$$
$\underline \h$ is not abelian, since $\begin{smallmatrix}4\\3\end{smallmatrix}\longrightarrow\begin{smallmatrix}5\\4\\3\end{smallmatrix}$ is an epimorphism in $\underline \h$, but we can only find a short exact sequence $\begin{smallmatrix}4\\3\end{smallmatrix}\longrightarrow\begin{smallmatrix}5\\4\\3\end{smallmatrix}\longrightarrow \begin{smallmatrix}5\end{smallmatrix}$ where $\begin{smallmatrix}5\end{smallmatrix}$ lies in $\U$, not in $\s$.

\section{Almost split sequences}

In this section, we assume $\B$ is Krull-Schmidt.

\begin{defn}\label{as}
Let $(\mathcal A, \EE_{\mathcal A}, \mathfrak{s}_{\mathcal A})$ be an arbitrary Krull-Schmidt extriangulated category. An $\EE_{\mathcal A}$-triangle $A\xrightarrow{x} B\xrightarrow{y} C\dashrightarrow$ is called an almost split sequence if the following conditions are satisfied:
\begin{itemize}
\item[(a)] This $\EE_{\mathcal A}$-triangle is nonsplit.
\item[(b)] If $a: A\to M$ is not a section, then there is a morphism $m:B\to M$ such that $a=mx$ (we call such $x$ left almost split).
\item[(c)] If $c:N\to C$ is not a retraction, then there is a morphism $n:N\to B$ such that $c=yn$ (we call such $y$ right almost split).
\end{itemize}
\end{defn}

\begin{rem}
Note that $A\xrightarrow{x} B\xrightarrow{y} C\overset{\delta}\dashrightarrow$ is an almost split sequence if and only if $\delta$ is an almost split extension defined in \cite[Definition 2.1]{INP}. Our definition of almost split sequence also coincides with Auslander-Reiten $\EE$-triangle defined in \cite[Definition 4.1]{ZZ1}. Particularly, the almost split sequences in $(\underline \h,\underline {\mathfrak{F}})$ coincides with the Auslander-Reiten sequences defined in \cite[Definition 4.6]{Sh}.
\end{rem}

\begin{rem}
Let $k$ be a field. If $\A=\mod \Lambda$ where $\Lambda$ is a finite dimensional $k$-algebra, then the almost split sequences defined in Definition \ref{as} are just the Auslander-Reiten sequences in $\mod \Lambda$. If $\A$ is a Krull-Schmidt, Hom-finite, $k$-linear triangulated category with a Serre functor, then the almost split sequences defined in Definition \ref{as} are just the Auslander-Reiten triangles in $\A$.
\end{rem}

According to \cite[Proposition 2.5]{INP} and \cite[Remark 4.2]{ZZ1}, we have the following lemma.

\begin{lem}
Let $(\mathcal A, \EE_{\mathcal A}, \mathfrak{s}_{\mathcal A})$ be a Krull-Schmidt extriangulated category. Let $A\xrightarrow{x} B\xrightarrow{y} C\dashrightarrow$ be an almost split sequence. Then $A,C$ are indecomposable objects.
\end{lem}

We recall the following basic property for almost split sequences.

\begin{lem}{\rm \cite[Proposition 2.9]{INP}}\label{min}
Let $(\mathcal A, \EE_{\mathcal A}, \mathfrak{s}_{\mathcal A})$ be a Krull-Schmidt extriangulated category. Let $A\xrightarrow{x} B\xrightarrow{y} C\dashrightarrow$ be an almost split sequence. Then $x$ is left minimal and $y$ is right minimal.
\end{lem}

\begin{proof}
Let $b:B\to B$ be a morphism such that $bx=x$. Then we have the following commutative diagram
$$\xymatrix{
A \ar[r]^x \ar@{=}[d] &B \ar[d]^b \ar[r]^y &C \ar@{-->}[r] \ar[d]^c &\\
A \ar[r]_x &B \ar[r]_y &C \ar@{-->}[r]  &.
}
$$
If $c$ is an isomorphism, by \cite[Corollary 3.6]{NP}, $b$ is an isomorphism. Then $x$ is left minimal. If $c$ is not an isomorphism, since $\mathcal A$ is Krull-Schmidt and $C$ is indecomposable, $1-c$ becomes an isomorphism. From the following commutative diagram
$$\xymatrix{
A \ar[r]^x \ar[d]_0 &B \ar[d]^{1-b} \ar[r]^y &C \ar@{-->}[r] \ar[d]^{1-c} &\\
A \ar[r]_x &B \ar[r]_y &C \ar@{-->}[r]  &
}
$$
we obtain that $1-c$ factors through $y$, but this means $A\xrightarrow{x} B\xrightarrow{y} C\dashrightarrow$ splits, a contradiction. Hence $c$ is an isomorphism.
\end{proof}

\begin{rem}
Since we assume $\A$ is Krull-Schmidt, our definition of an almost split sequence is a special case of \cite[Definition 1.3]{Liu}.
\end{rem}

Let $\C,\D$ be any two subcategories of $\B$. Denote by $\C_{ \D}$ the full
subcategory of $\C$ consisting of all objects with no
non-zero direct summands from  $\D$.

\begin{prop}\label{BF}
Let $A\xrightarrow{x} B\xrightarrow{y} C\dashrightarrow$ be an almost spilt sequence in $(\B,\EE,\mathfrak{s})$ with $A,C\in \h_{ \W}$. Then this almost spilt sequence is an $\mathbb{F}$-triangle, hence is also an almost spilt sequence in $(\h,\mathbb{F},\mathfrak{s}|_{\mathbb{F}})$.
\end{prop}

\begin{proof}
Since $A\in \h_{ \W}$, any morphism from $A$ to an object in $\W$ can not be a section. Hence by definition of almost split sequence, $x$ is $\W$-monic. By the similar argument we can get that $y$ is $\W$-epic. Then by Lemma \ref{exten}, $A\xrightarrow{x} B\xrightarrow{y} C\dashrightarrow$ is an $\mathbb{F}$-triangle.
\end{proof}

\cite[Proposition 5.11]{INP} shows the existence of almost split sequences in the ideal quotient category of an extriangulated category, if the original category has almost split sequences. For the convenience of the readers, we prove the following proposition, which can be realized as a special case of \cite[Proposition 5.11]{INP}.

\begin{prop}\label{P1}
Let $A\xrightarrow{x} B\xrightarrow{y} C\dashrightarrow$ be an almost spilt sequence in $(\h,\mathbb{F},\mathfrak{s}|_{\mathbb{F}})$. Then $A\xrightarrow{\underline x} B\xrightarrow{\underline y} C$ is an almost spilt sequence in $(\underline \h,\underline {\mathfrak{F}})$.
\end{prop}

\begin{proof}
By Corollary \ref{serre}, we have $A,C\in \h_{ \W}$ and $B\notin \W$. Hence $A\xrightarrow{\underline x} B\xrightarrow{\underline y} C$ is nonsplit in $(\underline \h,\underline {\mathfrak{F}})$. By Definition \ref{as} we can get that $\underline x$ is left almost split, since $x$ is left almost split. By the same reason we know $\underline y$ is right almost split. Then by definition $A\xrightarrow{\underline x} B\xrightarrow{\underline y} C$ is an almost spilt sequence in $(\underline \h,\underline {\mathfrak{F}})$.
\end{proof}

\begin{defn}\label{eq2}
Let $(\mathcal A, \EE_{\mathcal A}, \mathfrak{s}_{\mathcal A})$ be an arbitrary Krull-Schmidt extriangulated category. Two $\EE_{\mathcal A}$-triangles
$$A\xrightarrow{x} B\xrightarrow{y} C\dashrightarrow,\quad A\xrightarrow{x'} B' \xrightarrow{y'} C\dashrightarrow$$
are said to be equivalent if we have the following commutative diagram of $\EE_{\mathcal A}$-triangles.
$$\xymatrix{
A \ar[r]^-{x} \ar@{=}[d] &B \ar[r]^-{y} \ar[d]^{\cong} &C\ar@{-->}[r] \ar[d]^{\cong} &\\
A \ar[r]_-{x'} &B' \ar[r]_-{y'} &C\ar@{-->}[r] &
}
$$
\end{defn}

\begin{prop}\label{P2}
Let $A,B,C\in \h_{ \W}$ and $A\xrightarrow{\underline x} B\xrightarrow{\underline y} C$ be a short exact sequence in $(\underline \h,\underline {\mathfrak{F}})$. Then
\begin{itemize}
\item[(a)] There is an $\mathbb{F}$-triangle
$$\xymatrix{A \ar[r]^-{x'} &B' \ar[r]^-{y'} &C\ar@{-->}[r] &}$$
such that the following diagram is a commutative diagram of short exact sequences in $(\underline \h,\underline {\mathfrak{F}})$.
$$\xymatrix@C=1cm@R=0.3cm{
&B' \ar[dr]^{\underline {y'}} \ar[dd]^{\cong}\\
A \ar[dr]_{\underline x} \ar[ur]^{\underline {x'}} &&C\\
&B \ar[ur]_{\underline y}
}
$$
(we call $\mathbb{F}$-triangle $A\xrightarrow{x'} B'\xrightarrow{y'} C\dashrightarrow$ a pre-image of $A\xrightarrow{\underline x} B\xrightarrow{\underline y} C.$)
\item[(b)] Assume $A\xrightarrow{\underline x} B\xrightarrow{\underline y} C$ is an almost split sequence in $(\underline \h,\underline {\mathfrak{F}})$, then its pre-image $A\xrightarrow{x'} B'\xrightarrow{y'} C\dashrightarrow$ is an almost split sequence in $(\h,\mathbb{F},\mathfrak{s}|_{\mathbb{F}})$.
\item[(c)] Under the assumption in {\rm (b)}, if there is another $\mathbb{F}$-triangle
$$\xymatrix{A \ar[r]^-{x''} &B'' \ar[r]^-{y''} &C\ar@{-->}[r] &}$$
which induces an equivalence between short exact sequences in $(\underline \h,\underline {\mathfrak{F}})$:
$$\xymatrix@C=1cm@R=0.3cm{
&B \ar[r]^{\underline y} \ar[dd]^{\cong} &C\ar[dd]^{\cong}\\
A \ar[dr]_{\underline {x''}} \ar[ur]^{\underline x} \\
&B'' \ar[r]_{\underline {y''}} &C
}
$$
then we have an equivalence between $\mathbb{F}$-triangles:
$$\xymatrix{
A \ar[r]^-{x'} \ar@{=}[d] &B' \ar[r]^-{y'} \ar[d]^{\cong} &C\ar@{-->}[r] \ar[d]^{\cong} &\\
A \ar[r]_-{x''} &B'' \ar[r]_-{y''} &C\ar@{-->}[r] &
}
$$
\end{itemize}
\end{prop}

\begin{proof}
(a) By Theorem \ref{ck}, $A\xrightarrow{\underline x} B\xrightarrow{\underline y} C$ admits an $\mathbb{F}$-triangle
$$A\oplus W_1\xrightarrow{\alpha=\left(\begin{smallmatrix}x&a_2\\a_1&a_3\end{smallmatrix}\right)} B\oplus W_2 \xrightarrow{\left(\begin{smallmatrix}y&b_2\\b_1&b_3\end{smallmatrix}\right)} C\oplus W_3\dashrightarrow$$
with $W_1,W_2,W_3\in \W$. Since $\alpha$ is $\W$-monic, we have the following commutative diagram:
$$\xymatrix{A\oplus W_1\ar[r]^-{\left(\begin{smallmatrix}x&a_2\\a_1&a_3\end{smallmatrix}\right)} \ar[d]_-{\svech{0}{1}} &B\oplus W_2 \ar[dl]^-{\svech{r_1}{r_2}} \ar[r]^-{\left(\begin{smallmatrix}y&b_2\\b_1&b_3\end{smallmatrix}\right)} &C\oplus W_3 \ar@{-->}[r] &\\
W_1}$$
Then we have the following commutative diagram.
$$\xymatrix{
W_1 \ar[d]_-{\svecv{0}{1}} \ar@{=}[r] &W_1 \ar[d]^-{\svecv{a_2}{a_3}}\\
A\oplus W_1\ar[r]^-{\left(\begin{smallmatrix}x&a_2\\a_1&a_3\end{smallmatrix}\right)} \ar[d]_-{\svech{1}{0}} &B\oplus W_2  \ar[d]_-{\svech{b_1'}{b_2'}} \ar[r]_-{\left(\begin{smallmatrix}y&b_2\\b_1&b_3\end{smallmatrix}\right)} &C\oplus W_3 \ar@{=}[d] \ar@{-->}[r] &\\
A \ar@{-->}[d] \ar[r]_{b_1'x} &B_1 \ar@{-->}[d] \ar[r]_-{\svecv{c_1'}{c_2'}} &C\oplus W_3 \ar@{-->}[r] &\\
&&
}
$$
Since $\svech{r_1}{r_2} \circ \svecv{a_2}{a_3}=1$, the second column splits. Hence $B\oplus W_2\cong B_1\oplus W_1$. By the proof of Proposition \ref{main1}, $b_1'x$ is $\W$-monic and $\svecv{c_1'}{c_2'}$ is $\W$-epic. By Lemma \ref{exten}, $B_1\in \h$. Then $\underline {b_1'}$ is an isomorphism in $\underline \h$. Since $y=c_1'b_1'$, $\underline {c_1'}=\underline y(\underline {b_1'})^{-1}$.
Now we can get the following commutative diagram.
$$\xymatrix{
A \ar[r]^{x'} \ar@{=}[d] &B' \ar[d]^{s'} \ar[r]^{y'} &C \ar@{-->}[r]  \ar[d]^-{\svecv{1}{0}} &\\
A  \ar[r]^{b_1'x} &B_1 \ar[r]_-{\svecv{c_1'}{c_2'}} \ar[d]_{c_2'} &C\oplus W_3 \ar[d]^-{\svech{0}{1}} \ar@{-->}[r] &\\
&W_3 \ar@{=}[r] \ar@{-->}[d] &W_3 \ar@{-->}[d]\\
&&
}
$$
By the proof of Proposition \ref{main1} and duality, $x'$ is $\W$-monic and $y'$ is $\W$-epic. Hence $A\xrightarrow{x'}B'\xrightarrow{y'}C\dashrightarrow$ is an $\mathbb{F}$-triangle.
Thus we have the following commutative diagram of short exact sequences in $(\underline \h,\underline {\mathfrak{F}})$.
$$\xymatrix@C=1cm@R=0.3cm{
&&B' \ar[drr]^{\underline {y'}} \ar[dd]^{(\underline {b_1'})^{-1}\underline {s'}=\underline {b'}}_{\cong}\\
A \ar[drr]_{\underline x} \ar[urr]^{\underline {x'}} &&&&C\\
&&B \ar[urr]_{\underline {y}}
}
$$

(b) Now assume that $A\xrightarrow{\underline x} B\xrightarrow{\underline y} C$ is an almost split sequence in $(\underline \h,\underline {\mathfrak{F}})$. By the proof of (a), $A\xrightarrow{\underline {x'}} B'\xrightarrow{\underline {y'}} C$ is also an almost split sequence in $(\underline \h,\underline {\mathfrak{F}})$. Hence $A\xrightarrow{x'} B'\xrightarrow{y'} C\dashrightarrow$ can not split. Let $D$ be an indecomposable object in $\h$ and $a:A\to D$ be any morphism which is not a section. If $D\in \W$, since $x'$ is $\W$-monic, $d$ factors through $x'$. If $D\notin \W$, since $\underline a$ is not a section, either, there is a morphism $d:B'\to D$ such that $\underline a=\underline {dx'}$. Thus $a-dx'$ factors through $\W$, then it factors through $x'$. Hence there is a morphism $d':B'\to D$ such that $a-dx'=d'x'$, $a=(d+d')x'$, which implies $x'$ is left almost split. Dually we can show that $y'$ is right almost split. Hence $A\xrightarrow{x'} B'\xrightarrow{y'} C\dashrightarrow$ is almost split in $(\h,\mathbb{F},\mathfrak{s}|_{\mathbb{F}})$.

(c) If there is another $\mathbb{F}$-triangle
$$\xymatrix{A \ar[r]^-{x''} &B'' \ar[r]^-{y''} &C\ar@{-->}[r] &}$$
which induces an equivalence between short exact sequences in $(\underline \h,\underline {\mathfrak{F}})$.
$$\xymatrix@C=1cm@R=0.3cm{
&B \ar[r]^{\underline y} \ar[dd]^{\underline {b''}}_{\cong} &C \ar[dd]^{\cong}\\
A \ar[dr]_{\underline {x''}} \ar[ur]^{\underline x} \\
&B'' \ar[r]_{\underline {y''}} &C
}
$$
since $\underline x$ is left minimal, by the proof of (b), we get that $x''$ is left minimal. Since $\underline {b''b'x'}=\underline {x''}$, denote $b''b'$ by $b$, $x''-bx'$ factors through $\W$. Since $x'$ is $\W$-monic, $x''-bx'$ factors through $x'$. Then there exists a morphism $t_1:B'\to B''$ such that $x''-bx'=t_1x'$. Hence we get a commutative diagram of $\mathbb{F}$-triangles
$$\xymatrix{
A \ar[r]^-{x'} \ar@{=}[d] &B' \ar[r]^-{y'} \ar[d]^{b+t_1} &C\ar@{-->}[r] \ar[d]^c &\\
A \ar[r]_-{x''} &B'' \ar[r]_-{y''} &C\ar@{-->}[r] &
}$$
By the similar argument, we can get the following commutative diagram of $\mathbb{F}$-triangles.
$$\xymatrix{
A \ar[r]^-{x''} \ar@{=}[d] &B'' \ar[r]^-{y''} \ar[d]^{e} &C\ar@{-->}[r] \ar[d] &\\
A \ar[r]_-{x'} &B' \ar[r]_-{y'} &C\ar@{-->}[r] &
}$$
Since $x'$ and $x''$ are left minimal, $e(b+t_1)$ and $(b+t_1)e$ are isomorphisms. Hence $b+t_1$ is an isomorphism, by \cite[Corollary 3.6]{NP}, $c$ is also an isomorphism.
\end{proof}

By Proposition \ref{P2}, we can get the following corollary immediately.

\begin{cor}\label{C1}
Let $A\xrightarrow{\underline x} B\xrightarrow{\underline y} C$ be an almost split sequence in $(\underline \h,\underline {\mathfrak{F}})$. Then
\begin{itemize}
\item[(1)]  If $A\xrightarrow{\underline x} B\xrightarrow{\underline y} C$ is the image of an $\mathbb{F}$-triangle $A\xrightarrow{x} B\xrightarrow{y} C\dashrightarrow$, then this $\mathbb{F}$-triangle is an almost split sequence in $(\h,\mathbb{F},\mathfrak{s}|_{\mathbb{F}})$.
\item[(2)] The pre-image of $A\xrightarrow{\underline x} B\xrightarrow{\underline y} C$ is unique up to equivalence.
\item[(3)] If there is an equivalence between short exact sequences in $(\underline \h,\underline {\mathfrak{F}})$:
$$\xymatrix@C=1cm@R=0.3cm{
&B \ar[r]^{\underline y} \ar[dd]^{\cong} &C\ar[dd]^{\cong}\\
A \ar[dr]_{\underline {x'}} \ar[ur]^{\underline x} \\
&B' \ar[r]_{\underline {y'}} &C
}
$$
then any pre-image of $A\xrightarrow{\underline x} B\xrightarrow{\underline y} C$ and any pre-image of $A\xrightarrow{\underline {x'}} B'\xrightarrow{\underline {y'}} C$ are equivalent.
\end{itemize}
\end{cor}

By \cite[Theorem 4.19]{Sh}, since $\underline \h$ is quasi-abelian, a short exact sequence $A\xrightarrow{\underline x} B\xrightarrow{\underline y} C$ in $(\underline \h, \underline {\mathfrak{F}})$ is almost split if and only if $\underline x$ is minimal left almost split (resp. $\underline y$ is minimal right almost split). Then we can get the following corollary.

\begin{cor}
An $\mathbb{F}$-triangle $A\xrightarrow{x} B\xrightarrow{y} C\dashrightarrow$ is almost split if and only if $x$ is minimal left almost split (resp. $y$ is minimal right almost split).
\end{cor}

\begin{proof}
We only discuss the condition associated with $x$, the other condition is by dual.

The ``only if" part is followed by Definition \ref{as} and Proposition \ref{min}. Now we show the ``if" part. By Proposition \ref{P2}, we only need to prove that if $x$ is minimal left almost split, so is $\underline x$. By definition, we can obtain that $\underline x$ is left almost split when $x$ is.

Let $b:B\to B$ be a morphism such that $\underline x=\underline {bx}$. Then $x-bx$ factors through $\W$. Thus there is a morphism $b':B\to B$ such that $b'x=x-bx$ (note that $b'$ factors through $\W$). Then $x=(b+b')x$. Since $x$ is left minimal, $b+b'$ is an isomorphism, hence $\underline b$ is an isomorphism, which implies that $\underline x$ is left minimal.
\end{proof}

Now we can conclude the following result.

\begin{thm}\label{main3}
There is a one-to-one correspondence (up to equivalences) between the almost split sequences in $(\h,\mathbb{F},\mathfrak{s}|_{\mathbb{F}})$ and the almost split sequences in $(\underline \h,\underline {\mathfrak{F}})$.
\end{thm}

\begin{proof}
The correspondences are the following:
\begin{align*}
\sigma: ~\text{almost split sequence } A\xrightarrow{\underline x} B\xrightarrow{\underline y} C~\Longrightarrow \text{ its pre-image (a fixed one)};\\
\pi: ~\text{almost split sequence } A\xrightarrow{x} B\xrightarrow{y} C\dashrightarrow~\Longrightarrow~A\xrightarrow{\underline x} B\xrightarrow{\underline y} C.\quad \quad \quad
\end{align*}
By Proposition \ref{P1}, Proposition \ref{P2} and Corollary \ref{C1}, $\sigma$ and $\pi$ are one-to-one up to equivalences.
\end{proof}

According to Theorem \ref{main3}, we have the following corollary.

\begin{cor}
$(\h,\mathbb{F},\mathfrak{s}|_{\mathbb{F}})$ has almost split sequences if and only if $(\underline \h,\underline {\mathfrak{F}})$ has almost split sequences.
\end{cor}

\begin{rem}
We claim that Theorem \ref{main3} has a more general version. Note that $(\h,\mathbb{F},\mathfrak{s}|_{\mathbb{F}})$ is an extriangulated category in which $\W$ is a subcategory of projective-injective objects. $(\underline \h,\underline {\mathfrak{F}})$ is the same as $(\underline \h,\underline {\mathbb{F}},\underline {\mathfrak{s}|_{\mathbb{F}}})$. Hence by \cite[Proposition 5.11]{INP} and the proof Proposition \ref{P2}, we can conclude the following result:
\begin{itemize}
\item Let $(\mathcal A, \EE_{\mathcal A}, \mathfrak{s}_{\mathcal A})$ be a Krull-Schmidt extriangulated category and $\mathfrak{D}$ be a subcategory of projective-injective objects of $\A$. Then there is a one-to-one correspondence (up to equivalences) between the almost split sequences in $(\A,\EE_{\A},\mathfrak{s}_{\A})$ and the almost split sequences in $(\A/\mathfrak{D},\EE_{\A}/\mathfrak{D},\mathfrak{s}_{\A}/\mathfrak{D})$ (one can find details of this notion in \cite[Corollary 5.9, Proposition 5.11]{INP}).
\end{itemize}
\end{rem}


\section{Applications}

 We give some applications of the main results of this article, which are associated with  a cluster tilting object. 
\medskip

Let $\C$ be a Krull-Schmidt, Hom-finite, $k$-linear triangulated category with shift factor {\rm [1]}. Let $T\in \C$ be a cluster tiling object. Denote $\add T$ by $\T$. Then $(\T,\T)$ is a cotorsion pair on $\C$ and the heart of $(\T,\T)$ is just $\C/\T$. Moreover, we have $\C/\T\simeq \mod\End_{\C}(T[-1])$ by \cite[Corollaries 4.5, 4.6]{KZ} (also see \cite[Corollary 6.5]{IY}). Hence $\C/\T$ has Auslander-Reiten sequences. Note that Auslander-Reiten triangles in $\C$ (see \cite{Ha} and \cite{RV}) and Auslander-Reiten sequences in $\C/\T$ are both almost split sequences defined in Definition \ref{as}. By the results in \cite{KZ}, an object in $\C/\T$ is injective (resp. projective) if and only if it belongs to $\T[1]$ (resp. $\T[-1]$). For any indecomposable object $X\notin (\T\cup \T[1])$ (resp. $Z\notin (\T\cup \T[-1])$), it admits an Auslander-Reiten sequence $X\xrightarrow{\underline a} Y\xrightarrow {\underline b} Z$ in $\C/\T$. By Proposition \ref{P2}, we can find a pre-image of it, which is an Auslander-Reiten triangle in $\C$. By Theorem \ref{main3}, we can conclude the following result.

\begin{prop}
There is a one-to-one correspondence (up to equivalences) between the Auslander-Reiten triangles $X\to Y\to Z\to X[1]$ in $\C$ with $X,Z\in \C_\T$ and the Auslander-Reiten sequences in $\C/\T$.
\end{prop}

 Additionally, we have made the following observation.

\begin{lem}\label{6-1}
Let $X,Z\in \C$ be indecomposable objects.
\begin{itemize}
\item[(a-1)] If $X\notin (\T\cup \T[1])$, then $X$ admits an Auslander-Reiten triangle $X\to Y'\to Z'\to X[1]$ where $Y'\notin \T$ and $Z'\notin \T$.
\item[(a-2)] If $X\in \T[1]$, then $X$ admits an Auslander-Reiten triangle $X\to Y'\to Z'\to X[1]$ with $Z'\in \T$.
\item[(b-1)] If $Z\notin (\T\cup \T[-1])$, then $Z$ admits an Auslander-Reiten triangle $X''\to Y''\to Z\to X[1]$, we can find that $Y''\notin \T$ and $X''\notin \T$.
\item[(b-2)] If $Z\in \T[-1]$, then $Z$ admits an Auslander-Reiten triangle $X''\to Y''\to Z\to X[1]$ with $X''\in \T$.
\end{itemize}
\end{lem}

\begin{proof}
We check (a-1) and (a-2), the other two is by dual.

If $X\notin (\T\cup \T[1])$, then in $\C/\T$ it admits an Auslander-Reiten sequence $0\to X\to Y\to Z'\to 0$. By Proposition \ref{P2}, $X$ admits an Auslander-Reiten triangle $X\to Y'\to Z'\to X[1]$ such that $Y'\notin \T$ and $Z'\notin \T$.

If $X\in \T[1]$, then $X[-2]\in \T[-1]$. Hence $X[-2]$ admits an Auslander-Reiten sequence $$0\to X[-2]\to Y_1\to Z_1\to 0$$ in $\C/\T$. By Proposition \ref{P2}, we can obtain an Auslander-Reiten triangle $X[-2]\to Y_1'\to Z_1\to X[-1]$ in $\C$. Then $X$ admits an Auslander-Reiten triangle $X\xrightarrow{a} Y_1[2]'\xrightarrow{b} Z_1[2]\to X[1]$.

If $Z_1[2]\notin \T$, by Proposition \ref{BF}, $0\to X\xrightarrow{\underline a} Y_1[2]'\xrightarrow{\underline b} Z_1[2]\to 0$ is an Auslander-Reiten sequence in $\C/\T$. But $X$ is injective in $\C/\T$, a contradiction. Hence $Z_1[2]\in \T$.
\end{proof}

By the proof of this lemma, we have the following corollary.

\begin{cor}\label{6-2}
Let $X\in \T$ be an indecomposable object. Then $X$ admits two Auslander-Reiten triangles
$$X\to Y'\to Z'\to X[1]~~\mbox{and}~~ X[-1]\to Z''\to Y''\to X$$
with $Z',Z''\in \C_\T$.
\end{cor}

We can illustrate our results by the following classic example.

\begin{exm}
Let $Q$ be  the quiver
$3\longrightarrow 2\longrightarrow 1$
and $\C=\C_Q$ be the cluster category of $Q$ \cite{BMRRT}. The Auslander-Reiten quiver of $\C$ is
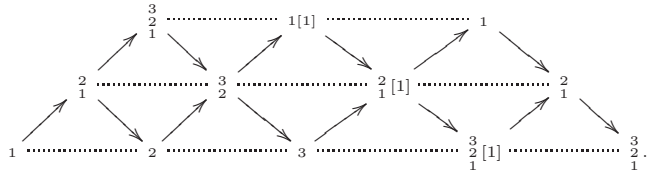
\begin{align}
	\tiny{\xymatrix @R=3mm @C=5mm{
			&&{\begin{smallmatrix}3\\2\\1\end{smallmatrix}}\ar[rd]&&
			{\begin{smallmatrix}1[1]\end{smallmatrix}}\ar[rd]\ar@{.}[ll]&&
			{\begin{smallmatrix}1\end{smallmatrix}}\ar[rd]\ar@{.}[ll]&&\\
			&{\begin{smallmatrix}2\\1\end{smallmatrix}}\ar[ru]\ar[rd]&&
			{\begin{smallmatrix}3\\2\end{smallmatrix}}\ar[ru]\ar[rd]\ar@{.}[ll]&&
			{\begin{smallmatrix}2\\1\end{smallmatrix}}[1]\ar[ru]\ar[rd]\ar@{.}[ll]&&
			{\begin{smallmatrix}2\\1\end{smallmatrix}}\ar[rd]\ar@{.}[ll]
			\\
			{\begin{smallmatrix}1\end{smallmatrix}}\ar[ru]&&\textnormal{$\begin{smallmatrix}2\end{smallmatrix}$}\ar[ru]\ar@{.}[ll]&&
			{\begin{smallmatrix}3\end{smallmatrix}}\ar[ru]\ar@{.}[ll]&&\textnormal{$\begin{smallmatrix}3\\2\\1\end{smallmatrix}[1]$}\ar[ru]\ar@{.}[ll]&&
			{\begin{smallmatrix}3\\2\\1\end{smallmatrix}}.\ar@{.}[ll]
	}}\notag
\end{align}
The direct sum $T=1[1]\oplus \begin{smallmatrix}2\\1\end{smallmatrix}[1]\oplus \begin{smallmatrix}3\\2\\1\end{smallmatrix}[1]$ gives a cluster tilting object. The Auslander-Reiten quiver of $\C/\T\simeq \mod\End_{\C}(T[-1])$ is the following:
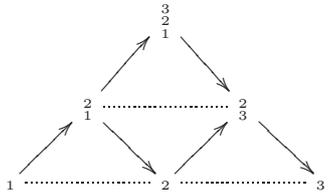
\begin{align}
	\tiny{\xymatrix @R=6mm @C=6mm{
			&&{\begin{smallmatrix}3\\2\\1\end{smallmatrix}}\ar[rd]&&\\
			&{\begin{smallmatrix}2\\1\end{smallmatrix}}\ar[ru]\ar[rd]&&
			{\begin{smallmatrix}2\\3\end{smallmatrix}}\ar[rd]\ar@{.}[ll]&
			\\
			{\begin{smallmatrix}1\end{smallmatrix}}\ar[ru]&&\textnormal{$\begin{smallmatrix}2\end{smallmatrix}$}\ar[ru]\ar@{.}[ll]&&
			{\begin{smallmatrix}3\end{smallmatrix}}.\ar@{.}[ll]
	}}\notag
\end{align}
\end{exm}
\vspace{2mm}

According to \cite[Theorem I.2.4]{RV}, $\C$ has a Serre functor if and only if $\C$ has Auslander-Reiten triangles. Hence by Lemma \ref{6-1} and Corollary \ref{6-2}, we can obtain the following theorem.

\begin{thm}
Let $\C$ be a Krull-Schmidt, Hom-finite, $k$-linear triangulated category. If $\C$ has a cluster tilting object, then $\C$ has a Serre functor.
\end{thm}

%

\vspace{1cm}

\hspace{-4mm}\textbf{Data Availability}\hspace{2mm} Data sharing not applicable to this article as no datasets were generated or analyzed during
the current study.
\vspace{2mm}

\hspace{-4mm}\textbf{Conflict of Interests}\hspace{2mm} The authors declare that they have no conflicts of interest to this work.

\end{document}